\documentclass[reqno]{amsart}
\usepackage{amsmath,amssymb}
\usepackage[mathscr]{euscript}
\usepackage{graphicx}
\usepackage{enumitem}
\usepackage[english]{babel}

\usepackage[small]{caption}

\usepackage{tikz}

\makeatletter
\@addtoreset{equation}{section}
\makeatother

\newtheorem{theorem}{Theorem}[section]
\newtheorem{lemma}[theorem]{Lemma}
\newtheorem{proposition}[theorem]{Proposition}

\newtheorem{remark}[theorem]{Remark}
\newtheorem{assumption}[theorem]{Assumption}

\numberwithin{equation}{section}
\newcommand{\mc}[1]{{\mathcal #1}}
\newcommand{\ms}[1]{{\mathscr #1}}
\newcommand{\mf}[1]{{\mathfrak #1}}
\newcommand{\mb}[1]{{\mathbf #1}}
\newcommand{\bb}[1]{{\mathbb #1}}

\newcounter{as}[section]

\newtheorem{asser}[as]{Assertion}

\newcommand{\<}{\langle}
\renewcommand{\>}{\rangle}
\renewcommand{\Cap}{{\rm cap}}

\begin{document}

\title[Metastability of the Blume-Capel model] {Metastability of the
  two-dimensional Blume-Capel model with zero chemical potential and
  small magnetic field}

\author{C. Landim, P. Lemire}

\address{\noindent IMPA, Estrada Dona Castorina 110, CEP 22460-320 Rio
  de Janeiro, Brasil. \newline
\noindent CNRS UMR 6085, Universit\'e de Rouen, Avenue de
  l'Universit\'e, BP.12, Technop\^ole du Madril\-let, F76801
  Saint-\'Etienne-du-Rouvray, France.  \newline e-mail: \rm
  \texttt{landim@impa.br} }

\address{\noindent Avenue de
  l'Universit\'e, BP.12, Technop\^ole du Madril\-let, F76801
  Saint-\'Etienne-du-Rouvray, France. 
\newline e-mail: \rm \texttt{paul.lemire@etu.univ-rouen.fr} }

\begin{abstract}
  We consider the two-dimensional Blume-Capel model with zero chemical
  potential and small magnetic field evolving on a large but finite
  torus.  We obtain sharp estimates for the transition time, we
  characterize the set of critical configurations, and we prove the
  metastable behavior of the dynamics as the temperature vanishes.
\end{abstract}

\maketitle

\section{Introduction}
\label{intro}

Since its first rigorous mathematical treatment \cite{lp71, fw98,
  cgov84, ov1, bh15}, metastability has been the subject of intensive
investigation from different perspectives \cite{ev06, bbf09, msv09,
  bg15, fmnss15}.

In \cite{begk1, begk2} Bovier, Eckhoff, Gayrard and Klein, BEGK from
now on, have shown that the potential theory of Markov chains can
outperform large deviations arguments and provides sharp estimates for
several quantities appearing in metastability, such as the expectation
of the exit time from a well or the probability to hit a configuration
before returning to the starting configuration.

Developing further BEGK's potential-theoretic approach, and with an
intensive use of data reduction through trace processes, Beltr\'an and
one of the authors of this paper, BL from now on, devised a scheme to
describe the evolution of a Markov chain among the wells, particularly
effective when the dynamics presents several valleys of the same depth
\cite{bl2, bl7, bl9}. The outcome of the method can be understood as a
model reduction through coarse-graining, or as the derivation of the
evolution of the slow variables of the chain.

In the case of finite state Markov chains \cite{bl4, lx15}, under
minimal assumptions, BL's method permits the identification of all
slow variables, the derivation of the time-scales at which they evolve
and the characterization of their asymptotic dynamics.

In contrast with the pathwise approach \cite{cgov84, ov1} and the
transition path theory \cite{ev06, msv09}, BEGK's and BL's approach do
not attempt to describe the tube of typical trajectories in a
transition between two valleys, nor does it identify the critical
configurations which are visited with high probability in such
transitions.

Nevertheless, under weak hypotheses, introduced in Section \ref{sec2}
below, potential-theoretic arguments together with data reduction
through trace processes provide elementary identities and estimates
which permit, without much effort, to characterize the critical
configurations, and to compute the sub-exponential pre-factors of the
expectation of hitting times.  The purpose of this paper is to
illustrate these assertions by examining the metastable behavior of
the Blume-Capel model.

The Blume-Capel model is a two dimensional, nearest-neighbor spin
system where the single spin variable takes three possible values:
$-1$, $0$ and $+1$. One can interpret it as a system of particles with
spins. The value $0$ of the spin at a lattice site corresponds to the
absence of particles, whereas the values $\pm 1$ correspond to the
presence of a particle with the respective spin.

The metastability of the Blume-Capel model has been investigated by
Cirillo and Olivieri \cite{co96}, Manzo and Olivieri \cite{mo01}, and
more recently by Cirillo and Nardi \cite{cn13}. 

We consider here a Blume-Capel model with zero chemical potential and
a small positive magnetic field. We examine its metastable behavior in
the zero-temperature limit in a large, but fixed, two-dimensional
square with periodic boundary conditions. In this case, there are two
metastable states, the configurations where all spins are equal to
$-1$ or all spins equal to $0$, and one ground state, the
configuration where all spins are equal to $+1$.

The main results state that starting from $\bf -1$, the configuration
where all spins are equal to $-1$, the chain visits $\bf 0$ before
hitting $\bf +1$. We also characterize the set of critical
configurations. These results are not new and appeared in \cite{co96,
  cn13}, but we present a proof which relies on a simple inequality
from the potential theory of Markov chains. We compute the exact
asymptotic values of the transition times, which corresponds to the
life-time of the metastable states. The previous results on the
transition time, based on the pathwise approach which relies on large
deviations arguments, presented estimates with exponential errors. To
complete the picture, we show that the expectation of the hitting time
of the configuration $\bf 0$ starting from $\bf -1$ is much larger
than the transition time. This phenomenon, which may seem
contradicting the fact that the chain visits $\bf 0$ before hitting
$\bf +1$, occurs because the main contribution to the expectation
comes from the event that the chain first hits $\bf +1$ and then
visits $\bf 0$. The very small probability of this event is
compensated by the very long time the chain remains at $\bf +1$.

Finally, we prove the metastable behavior of the Blume-Capel model in
the sense of BL. Let $\Sigma$ be the set of configurations and let
$\ms V_a$ be a neighborhood of the configuration $\bf a$, where $a=0,
\pm 1$. For example $\ms V_a = \{{\bf a}\}$. Fix a real number $\mf d\not\in
\{-1,0,+1\}$, and let $\phi: \Sigma \to \{-1,0,1, \mf d\}$ be the
projection defined by
\begin{equation*}
\phi(\sigma) \;=\; \sum_{a=-1,0,1} a \, \mb 1\{\sigma \in \ms
V_{\bf a}\} \;+\; \mf d \, \mb 1\Big\{\sigma \not\in \bigcup_{a=-1,0,+1} \ms
V_{\bf a} \Big\}\;.
\end{equation*}
Denote the inverse of the temperature by $\beta$.  We prove that there
exists a time scale $\theta_\beta$ for which
$\phi(\sigma(t\theta_\beta))$ converges, as $\beta\to\infty$, to a
Markov chain in $\{-1,0,+1\}$. The point $+1$ is an absorbing point
for this Markov chain, and the other jump rates are given by
\begin{equation*}
r(-1,0) \;=\;  r(0,1) \;=\; 1\;, \quad
r(-1,1) \;=\; r(0,-1)\;=\; 0\;.  
\end{equation*}
As we said above this result can be interpreted as a model reduction
by coarse-graining, or as the identification of a slow variable,
$\phi$, whose evolution is asymptotically Markovian.

The article is divided as follows. In Section \ref{not}, we state the
main results. In Section \ref{sec2}, we introduce the main tools used
throughout the article and we present general results on finite-state
reversible Markov chains. In section \ref{sec3}, we examine the
transition from ${\bf -1}$ to $\bf 0$, and in Section \ref{sec4} the
one from $\bf 0$ to $\bf +1$. In Section \ref{sec6}, we analyze the
hitting time of ${\bf 0}$ starting from $\bf -1$.  In the last
section, we prove the metastable behavior of the Blume-Capel model
with zero chemical potential as the temperature vanishes.

\section{Notation and main results}
\label{not}

Fix $L>1$ and let $\Lambda_L=\bb T_L\times\bb T_L$, where $\bb
T_L=\{1,\cdots,L\}$ is the discrete, one-dimensional torus of length
$L$.  Denote the configuration space by
$\Omega=\{-1,0,1\}^{\Lambda_L}$, and by the Greek letters $\sigma$,
$\eta$, $\xi$ the configurations of $\Omega$. Hence, $\sigma(x)$,
$x\in \Lambda_L$, represents the spin at $x$ of the configuration
$\sigma$.

Fix an external field $0 < h < 1$, and denote by $\mathbb H : \Omega
\to \bb R$ the Hamiltonian given by
\begin{equation}
\label{eq06}
\mathbb H (\sigma) \;=\; 
\sum \big(\sigma(y)-\sigma(x)\big)^2 \, -\, h\sum_{x\in
  \Lambda_L}\sigma(x)\;, 
\end{equation}
where the first sum is carried over all unordered pairs of
nearest-neighbor sites of $\Lambda_L$. Let $n_0 = [2/h]$, where $[a]$
represents the integer part of $a\in\bb R_+$. We assume that
$L>n_0+3$.

Denote by $\beta>0$ the inverse of the temperature and by $\mu_\beta$
the Gibbs measure associated to the Hamiltonian $\mathbb H$ at inverse
temperature $\beta$,
\begin{equation}
\label{mesgibbs}
\mu_\beta (\sigma)\, =\, \frac 1{Z_\beta} \, e^{-\beta \mathbb H (\sigma)} \, ,
\end{equation}
where $Z_\beta$ is the partition function, the normalization constant
which turns $\mu_\beta$ a probability measure.

Denote by ${\bf -1}$, $\textbf{0}$, $\textbf{+1}$ the
configurations of $\Omega$ with all spins equal to $-1$, $0$, $+1$,
respectively. These three configurations are local minima of the
energy $\bb H$, $\bb H(\textbf{+1})<\bb H(\textbf{0})<\bb
H({\bf -1})$, and $\textbf{+1}$ is the unique ground state.

The Blume-Capel dynamics is the continuous-time Markov chain on
$\Omega$, denoted by $\{\sigma_t \; :\ \ t\ge 0\}$, whose
infinitesimal generator $L_\beta$ acts on functions $f: \Omega\to
\bb R$ as
\begin{equation}
\label{gen1}
\begin{aligned}
(L_\beta f) (\sigma) \, &=\, 
\sum_{x \in \Lambda_L} R_\beta(\sigma,\sigma^{x,+}) \big[ f(\sigma^{x,+}) - f(\sigma) \big] \\
& +\, \sum_{x \in \Lambda_L} R_\beta(\sigma,\sigma^{x,-}) \big[ f(\sigma^{x,-}) - f(\sigma) \big] \, .
\end{aligned}
\end{equation}
In this formula, $\sigma^{x,\pm}$ represents the configuration obtained
from $\sigma$ by modifying the spin at $x$ as follows,
\begin{equation*}
\sigma^{x,\pm} (z) := 
\begin{cases}
\sigma (x) \pm 1 \ \textrm{mod}\ 3 & \textrm{ if \ } z=x\\
\sigma (z) & \textrm{ if \ } z\neq x\, ;
\end{cases}
\end{equation*}
and the rates $R_\beta$ are given by
\begin{equation*}
R_\beta(\sigma,\sigma^{x,\pm}) \, =\, 
\exp\Big\{ -\beta \big[\mathbb H (\sigma^{x ,\pm}) 
- \mathbb H (\sigma) \big]_+\Big\}\, , \quad x\in \Lambda_L\, ,
\end{equation*}
where $a_+$, $a\in \bb R$, stands for the positive part of $a$: $a_+ =
\max\{a, 0\}$.  

Clearly, the Gibbs measure $\mu_\beta$ satisfies the detailed balance
condition
\begin{equation*}
\mu_\beta (\sigma) R_\beta(\sigma,\sigma^{x,\pm})\, =\, 
\min\big\{\mu_\beta (\sigma)\, ,\, \mu_\beta (\sigma^{x, \pm}) \big\}
\, =\, \mu_\beta (\sigma^{x,\pm}) R_\beta(\sigma^{x,\pm},\sigma)\, ,
\end{equation*}
and is therefore reversible for the dynamics.

Denote by $D(\bb R_+, \Omega)$ the space of right-continuous functions
${\bf x}: \bb R_+ \to \Omega$ with left-limits and by $\bb
P_\sigma=\bb P^\beta_\sigma$, $\sigma\in \Omega$, the probability
measure on the path space $D(\bb R_+, \Omega)$ induced by the Markov
chain $\sigma_t$ starting from $\sigma$. Expectation with respect to
$\bb P_\sigma$ is represented by $\bb E_\sigma$.

Denote by $H_{\ms A}$, $H^+_{\ms A}$, ${\ms A}\subset \Omega$, the
hitting time and the time of the first return to ${\ms A}$,
respectively:
\begin{equation}
\label{201}
H_{\ms A} \;=\; \inf \big \{t>0 : \sigma_t \in {\ms A} \big\}\;,
\quad
H^+_{\ms A} \;=\; \inf \big \{t>\tau_1 : \sigma_t \in {\ms A} \big\}\; ,  
\end{equation}
where $\tau_1$ represents the time of the first jump of the chain
$\sigma_t$.  We sometimes write $H({\ms A})$, $H^+({\ms A})$ instead
of $H_{\ms A}$, $H^+_{\ms A}$.

\begin{proposition}
\label{0b1}
Starting from $\bf -1$ the chain visits the state $\bf 0$ in its way
to the ground state $\bf +1$.
\begin{equation*}
\lim_{\beta\to\infty} 
\bb P_{{\bf -1}} [H_{\textbf{+1}} < H_{\textbf{0}}]  \;=\; 0\;.
\end{equation*}
\end{proposition}

Recall the definition of $n_0$ introduced just below \eqref{eq06}.
Denote by $\mf R^l$ (resp. $\mf R^l_0$) the set of configurations in
$\{-1,0, +1\}^{\Lambda_L}$ in which there are $n_0(n_0+1)+1$ spins
which are not equal to $-1$ (resp. $0$). Of these spins, $n_0(n_0+1)$
form a $n_0\times (n_0+1)$-rectangle of $0$ spins (resp. $+1$
spins). The remaing spin not equal to $-1$ is equal to $0$
(resp. $+1$) and is attached to the longest side of the rectangle
(cf. configuration $\sigma'$ in Figure \ref{fig1} with $n_0=5$). All
configurations in $\mf R^l$ have the same energy, as well as all
configurations in $\mf R^l_0$.

The next result states that, starting from $\bf -1$, the chain reaches
the set $\mf R^l$ before hitting $\bf 0$.

\begin{proposition}
\label{mt4}
We have that
\begin{equation*}
\lim_{\beta\to\infty} \bb P_{{\bf -1}} [H_{\mf R^l} < H_{\bf 0}]
\;=\; 1\;, \quad
\lim_{\beta\to\infty} \bb P_{{\bf 0}} [H_{\mf R^l_0} < H_{\bf +1}]\;.
\end{equation*}
\end{proposition}

Denote by $\lambda_\beta(\sigma)$, $\sigma\in\Omega$, the holding
rates of the Markov chain $\sigma_t$, and by $p_\beta(\eta,\xi)$,
$\eta$, $\xi\in \Omega$, the jump probabilities, so that
$R_\beta(\eta,\xi) = \lambda_\beta(\eta) p_\beta(\eta,\xi)$. Let
$M_\beta(\eta)=\mu_\beta(\eta) \lambda_\beta(\eta)$ be the stationary
measure for the embedded discrete-time Markov chain.

Denote by $\Cap (\ms A, \ms B)$ the capacity between two disjoint
subsets $\ms A$, $\ms B$ of $\Omega$: 
\begin{equation}
\label{202}
\Cap (\ms A, \ms B) \;=\; \Cap_\beta (\ms A, \ms B) 
\;=\; \sum_{\sigma\in\ms A} M_\beta(\sigma)\,
\bb P_\sigma [H_{\ms B} < H_{\ms A}^+]\;,
\end{equation}
and let 
\begin{equation}
\label{203}
\theta_\beta \;=\;  \frac{\mu_\beta({\bf -1})} 
{\Cap ({\bf -1},\{\textbf{0}, {\bf +1}\}) } 
\end{equation}
be the time-scale in which the Blume-Capel model reaches the ground
state $\bf +1$ starting from the local minima $\bf -1$ or $\bf 0$.

\begin{proposition}
\label{mt3}
For any configuration $\eta\in\mf R^l$ and any configuration
$\xi\in\mf R^l_0$, 
\begin{equation*}
\lim_{\beta\to\infty} \frac{\Cap ({\bf -1},\{{\bf 0}, {\bf +1}\})}
{\mu_\beta (\eta)}\;=\;\frac{4(2n_0+1)}3\, |\Lambda_L|\;=\; 
\lim_{\beta\to\infty} \frac{\Cap ({\bf 0},\{{\bf -1}, {\bf +1}\})}
{\mu_\beta (\xi)}\;.
\end{equation*}
\end{proposition}

The first identity of this proposition is proved in Section \ref{sec3}
and the second one in Section \ref{sec4}. 

\begin{proposition}
\label{mt2}
The expected time to visit the ground state starting from ${\bf -1}$
and from ${\bf 0}$ are given by
\begin{equation*}
\lim_{\beta\to\infty} \frac 1{\theta_\beta} 
\, \bb E_{\bf -1} [ H_{\bf +1}] \; = \; 2\;, \quad
\lim_{\beta\to\infty} \frac 1{\theta_\beta} \,
 \bb E_{\bf 0} [ H_{\bf +1}] \; =\; 1\;.
\end{equation*}
\end{proposition}

We have seen in Proposition \ref{0b1} that starting from $\bf -1$ the
process reaches $\bf 0$ before visiting $\bf +1$. In contrast, the
next identity shows that the main contribution to the expectation $\bb
E_{\bf -1} [ H_\textbf{0}]$ comes from the event in which the process,
starting from $\bf -1$, first visits $\bf +1$, remains there for a
very long time and then reaches $\bf 0$. We have that
\begin{equation}
\label{eq04}
\frac 1{\theta_\beta} \, 
\bb E_{\bf -1} [ H_\textbf{0}] \; = \; \big(b +o(1)\big) \, 
\frac{\mu_\beta(\textbf{+1})}{\mu_\beta(\textbf{0})}\,
\bb P_{{\bf -1}} [H_\textbf{+1} <H_{\textbf{0}}] \;,
\end{equation}
where $o(1)$ is an expression which vanishes as $\beta\uparrow\infty$,
and
\begin{equation}
\label{eq03}
\lim_{\beta\to\infty}  \frac{\mu_\beta(\textbf{+1})}{\mu_\beta(\textbf{0})}\,
\bb P_{{\bf -1}} [H_\textbf{+1} <H_{\textbf{0}}] \;=\; \infty\;.
\end{equation}

A self-avoiding path $\gamma$ from $\ms A$ to $\ms B$, $\ms A$, $\ms
B\subset \Omega$, $\ms A\cap \ms B = \varnothing$, is a sequence of
configurations $(\xi_0, \xi_1, \dots, \xi_n)$ such that $\xi_0\in \ms
A$, $\xi_n\in \ms B$, $\xi_j\not \in \ms A\cup \ms B$, $0<j<n$, $\xi_i
\not = \xi_j$, $i\not =j$, $R_\beta(\xi_i,\xi_{i+1})>0$, $0\le i
<n$. Denote by $\Gamma_{\ms A, \ms B}$ the set of self-avoiding paths
from $\ms A$ to $\ms B$ and let
\begin{equation}
\label{204}
\bb H(\ms A, \ms B) \;:=\; \min_{\gamma\in \Gamma_{\ms A, \ms B}} \bb
H (\gamma)\;, \quad
\bb H(\gamma) \;:=\; \max_{0\le i\le n} \bb H (\xi_i)\;.
\end{equation}  

Let $\ms M = \{{\bf -1}, {\bf 0}, {\bf +1}\}$ be the set of ground
configurations of the main wells, and let $\ms V_{\eta}$, $\eta\in \ms
M$, be a neighborhood of the configuration $\eta$. We assume that all
configurations $\sigma\in\ms V_\eta$, $\sigma\not = \eta$, fulfill the
conditions
\begin{equation}
\label{cond}
\bb H(\sigma) \;>\; \bb H(\eta)   \;, \quad
\bb H(\eta,\sigma) \,-\, \bb H(\eta) \;<\; \bb H({\bf -1}, \{ {\bf 0},
{\bf +1}\}) \,-\, \bb H({\bf -1})
\;.
\end{equation}
The right hand side in the second condition represents the energetic
barrier the chain needs to surmount to reach the set $\{{\bf 0}, {\bf
  +1}\}$ starting from $\bf -1$, while the left hand side represents
the energetic barrier to go from $\eta$ to $\sigma$.

It follows from Proposition \ref{mt3} that 
\begin{equation}
\label{200}
\bb H ({\bf -1},\{{\bf 0}, {\bf +1}\}) -
\bb H({\bf -1}) \;=\; \bb H ({\bf 0},\{{\bf -1}, {\bf +1}\})
- \bb H({\bf 0})\;.
\end{equation}
We may therefore replace the expression on the right hand side of
\eqref{cond} by the one on the right hand side of the previous
formula.

Clearly, $\ms V_\eta = \{\eta\}$, $\eta\in \ms M$, is an example of
neighborhoods satisfying \eqref{cond}.  Let $\ms V$ be the union of
the three neighborhoods, $\ms V = \cup_{\eta\in\ms M} \ms V_\eta$, and
let $\pi : \ms M \to \{-1,0,1\}$ be the application which provides the
magnetization of the states $\bf -1$, $\bf 0$, $\bf +1$: $\pi({\bf
  -1}) = -1$, $\pi({\bf 0}) = 0$, $\pi({\bf +1}) = 1$.  Denote by
$\Psi = \Psi_{\ms V} : \Omega\to \{-1,0, 1, [\beta]\}$ the projection
defined by $\Psi(\sigma) =\pi(\eta)$ if $\sigma\in \ms V_\eta$,
$\Psi(\sigma) = [\beta]$, otherwise:
\begin{equation*}
\Psi(\sigma) \;=\; \sum_{\eta \in \ms M} \pi(\eta) \, \mb 1\{\sigma \in \ms
V_\eta\} \;+\; [\beta] \, \mb 1\Big\{\sigma \not\in \bigcup_{\eta\in\ms M} \ms
V_\eta \Big\}\;.
\end{equation*}
Recall from \cite{bl4} the definition of the soft topology.

\begin{theorem}
\label{mt1b}
The speeded-up, hidden Markov chain $X_\beta(t) =
\Psi\big(\sigma(\theta_\beta t)\big)$ converges in the soft topology
to the continuous-time Markov chain $X(t)$ on $\{-1,0, 1\}$ in which
$1$ is an absorbing state, and whose jump rates are given by
\begin{equation*}
r(-1,0) \;=\;  r(0,1) \;=\; 1\;, \quad
r(-1,1) \;=\; r(0,-1)\;=\; 0\;. 
\end{equation*}
\end{theorem}

\begin{remark}
\label{rm3}
Denote by $\ms B_{\eta}$, $\eta\in \ms M$, the basin of attraction of
$\eta$:
\begin{equation*}
\ms B_{\eta} \;=\; \{\sigma : \lim_{\beta\to\infty} 
\bb P_\sigma[H_{\ms M \setminus \{\eta\}} < H_{\eta}] =0 \}\;.
\end{equation*}
We prove in \eqref{basin} that $\ms V_\eta\subset \ms B_{\eta}$.
\end{remark}

\section{Metastability of reversible Markov chains}
\label{sec2}

We present in this section some results on reversible Markov chains.
Consider two nonnegative sequences $(a_N : N\ge 1)$, $(b_N : N\ge
1)$. The notation $a_N \prec b_N$ (resp. $a_N \preceq b_N$) indicates
that $\limsup_{N\to\infty} a_N/b_N= 0$ (resp. $\limsup_{N\to\infty}
a_N/b_N < \infty$), while $a_N \approx b_N$ means that $a_N \preceq
b_N$ and $b_N \preceq a_N$.

A set of nonnegative sequences $(a^r_N: N\ge 1)$, $r\in \mf R$, is
said to be \emph{ordered} if for all $r\not = s\in\mf R$ $\arctan
(a^r_N / a^s_N)$ converges.

Fix a finite set $E$.  Consider a sequence of continuous-time,
$E$-valued Markov chains $\{\eta^N_t : t\ge 0\}$, $N\ge 1$.  We
assume, throughout this section, that the chain $\eta^N_t$ is
\emph{irreducible}, that the unique stationary state, denoted by
$\mu_N$, is \emph{reversible}, and that the jump rates of the chain
$\eta^N_t$, denoted by $R_N(x,y)$, $x\not = y\in E$, satisfy the
following hypothesis. Let $\bb Z_+ = \{0, 1, 2, \dots \}$, let $\bb B$
be the bonds of $E$: $\bb B = \{(x,y) \in E\times E: x\not = y\in
E\}$, and let $\mf A_m$, $m\ge 1$, be the set of functions $k:\bb B
\to \bb Z_+$ such that $\sum_{(x,y)\in\bb B} k(x,y) =m$.

\begin{assumption}
\label{mhyp} 
We assume that for every $m\ge 1$ the set of sequences
\begin{equation*}
\Big\{ \prod_{(x,y)\in\bb B} R_N(x,y)^{k(x,y)} : N \ge 1\Big\}
\;,\quad k\in\mf A_m
\end{equation*}
is ordered.
\end{assumption}

This assumption is slightly weaker than the hypotheses (2.1), (2.2) in
\cite{bl4}, but strong enough to derive all results presented in that
article.  It is also not difficult to check that the Blume-Capel model
introduced in the previous section fulfills Assumption \ref{mhyp}.

We adopt here similar notation to the one introduced in the previous
section. For example, $\lambda_N(x)$ represents the holding rates of
the Markov chain $\eta^N_t$, $p_N(x,y)$, $x$, $y\in E$, the jump
probabilities, and $M_N(x)=\mu_N(x) \lambda_N(x)$ a stationary measure
of the embedded discrete-time Markov chain.  Analogously, $\bb P_x=\bb
P^N_x$, $x\in E$, represents the distribution of the Markov chain
$\eta^N_t$ starting from $x$ and $\bb E_x$ the expectation with
respect to $\bb P_x$.

Denote by $H_A$ (resp. $H^+_A$), $A\subset E$, the hitting time of
(resp. the return time to) the set $A$, introduced in \eqref{201}, and
by $\Cap (A,B)$ the capacity between two disjoint subsets $A$, $B$ of
$E$, as defined in \eqref{202}.

The following identity will be used often. Let $A$, $B$ be two
disjoint subsets of $E$ and let $x$ be a point which does not belong
to $A\cup B$. We claim that
\begin{equation}
\label{mf01}
\bb P_x [ H_A < H_B] \;=\; \frac{\bb P_x [ H_A <
  H^+_{B\cup\{x\}}]}
{\bb P_x [ H_{A\cup B} < H^+_x ]}\;\cdot
\end{equation}
To prove this identity, intersect the event $\{H_A<H_B\}$ with the set
$\{H^+_x < H_{A\cup B}\}$ and its complement, and then apply the
strong Markov property to get that
\begin{equation*}
\bb P_x [ H_A < H_B] \;=\; \bb P_x [ H^+_x < H_{A\cup B}] 
\, \bb P_x [ H_A < H_B]  \;+\; \bb P_x [ H_A <
H^+_{B\cup\{x\}}]\; .
\end{equation*}
To obtain \eqref{mf01} it remains to subtract the first term on the
right hand side from the left hand side. 

Multiply and divide the right hand side of \eqref{mf01} by $M(x)$
and recall the definition of the capacity to obtain that
\begin{equation}
\label{mf05}
\bb P_x [ H_A < H_B] \;=\; \frac{M(x) \bb P_x [ H_A <
  H^+_{B\cup\{x\}}]} {\Cap (x, A\cup B)}
\;\le\; \frac{M(x) \bb P_x [ H_A <
  H^+_{x}]} {\Cap (x, A\cup B)}\;\cdot
\end{equation}
Hence, by definition of capacity and since, by \cite[Lemma 2.2]{gl14},
the capacity is monotone.
\begin{equation}
\label{mf02}
\bb P_x [ H_A < H_B] 
\;\le\; \frac{\Cap (x, A)} {\Cap (x, A\cup B)}
\;\le\; \frac{\Cap (x, A)} {\Cap (x, B)}\;\cdot
\end{equation}
\smallskip

For any disjoint subsets $A$ and $B$ of $E$,
\begin{equation}
\label{mf03}
\Cap (A, B) \;\approx\; 
\max_{x\in A}  \max_{y\in B} \Cap (x, y)\;. 
\end{equation}
Indeed, on the one hand, by monotonicity of the capacity,
\begin{equation*}
\Cap (A, B) \;\ge\; \max_{x\in A} 
\Cap (x, B) \;\ge\; \max_{x\in A} \max_{y\in B} 
\Cap (x, y) \;.
\end{equation*}
On the other hand, by definition of the capacity,
\begin{equation*}
\Cap (A, B) \; =\; \sum_{y\in B} M(y) \bb P_y[
H_{A} < H_{B}^+] 
\;\le\; \sum_{x\in A}
\sum_{y\in B} M(y) \bb P_y[ H_{x} < H_{B}^+] \;,
\end{equation*}
Therefore,
\begin{equation}
\label{mf06}
\Cap (A, B) \; \le \; \sum_{x\in A} \Cap (x, B)
\;\le\; |A|\, \max_{x\in A} \Cap (x, B) \;,
\end{equation}
where $|A|$ stands for the cardinality of $A$. Repeating this
argument for $B$ in place of $A$, we conclude the proof of
\eqref{mf03}.

Let $G_N: E\times E\to \bb R_+$ be given by $G_N(x,y) = \mu_N(x)
R_N(x,y)$ and note that $G_N$ is symmetric. In the electrical network
interpretation of reversible Markov chains, $G_N(x,y)$ represents the
conductance of the bond $(x,y)$. Recall that a self-avoiding path
$\gamma$ from $A$ to $B$, $A$, $B\subset E$, $A\cap B = \varnothing$,
is a sequence of sites $(x_0, x_1, \dots, x_n)$ such that $x_0\in A$,
$x_n\in B$, $x_j\not \in A\cup B$, $0<j<n$, $x_i \not = x_j$, $i\not
=j$, $R_N(x_i,x_{i+1})>0$, $0\le i <n$. Denote by $\Gamma_{A,B}$ the
set of self-avoiding paths from $A$ to $B$ and let
\begin{equation*}
G_N(A,B) \;:=\; \max_{\gamma\in \Gamma_{A,B}} G_N(\gamma)\;, \quad
G_N(\gamma) \;:=\; \min_{0\le i<n} G_N(x_i,x_{i+1})\;.
\end{equation*}
By \cite[Lemma 4.2]{bl4}, for every disjoint subsets $A$, $B$ of $E$,
the limit
\begin{equation}
\label{303}
\lim_{N\to\infty} \frac {\Cap (A,B)}{G_N(A,B)}   \;\;\text{exists and
  belongs to $(0,\infty)$}\;.
\end{equation}

\begin{remark}
\label{rm2}
Suppose that the stationary state $\mu_N$ is a Gibbs measure
associated to an energy $\bb H$, and that we are interested in the
Metropolis dynamics: $\mu_N(x) = Z^{-1}_N \exp\{-N \bb H(x)\}$, where
$Z_N$ is the partition function, $R_N(x,y) = \exp\{-N [\bb H(y) - \bb
H(x)]_+\}$. In this context, $G_N(x,y) = Z^{-1}_N \exp\{-N \max [\bb
H(x), \bb H(y)] \} $. In particular, for a path $\gamma=(x_0, x_1,
\dots, x_n)$, $G_N(\gamma) = Z^{-1}_N \exp\{-N \max_i H(x_i)\}$, and
for two disjoint subsets $A$, $B$ of $E$,
\begin{align*}
G_N (A,B) \; &=\; \frac 1{Z_N} \exp\Big\{ -N \min_{\gamma\in \Gamma_{A,B}}
\max_{x\in\gamma} \bb H(x) \Big\} \\
& =\; \frac 1{Z_N} \exp\big\{ - N \bb H(x_{A,B})
\big\} \;=\; \mu_N(x_{A,B})\;, 
\end{align*}
where $x_{A,B}$ represents the configuration with highest energy in
the optimal path joining $A$ to $B$.
\end{remark}

\begin{lemma}
\label{mest1}
Let $E_1$ be a subset of $E$. Assume that for every $y\not
\in E_1 $, $z\in E_1$ such that $\mu(z) \preceq \mu(y)$, 
\begin{equation}
\label{ehyp1}
\frac {\Cap(y,z)}{\mu(z)} \;\prec\; \frac {\Cap(y , E_1)}{\mu(y)}
\;\cdot   
\end{equation}
Then, for any $B\subset E_1$, $x\in E_1\setminus B$,  
\begin{equation*}
\bb E_x [H_B] \;=\; \big (1 + o(1) \big) \, \frac{1}{\Cap(x, B)} 
\sum_{y\in E_1} \mu(y)  \, \bb P_y [ H_x < H_B ]\;.
\end{equation*}
\end{lemma}

\begin{proof}
By \cite[Proposition 6.10]{bl2},
\begin{equation*}
\bb E_x [H_B] \;=\; \frac{1}{\Cap(x,B)}
\sum_{y \in E} \mu(y) \, \bb P_y [H_{x} < H_B] \;.
\end{equation*}
Denote by $\bb P^1_z$, $z\in E_1$, the distribution of the trace of
$\sigma(t)$ on the set $E_1$ starting from $z$ (cf. \cite[Section
6]{bl2} for the definition of trace), and let $q(y,z) = \bb
P_y[H_{E_1} = H_z]$, $y \not \in E_1$, $z\in E_1$. Decomposing the
previous sum according to $y \in E_1$, $y \not \in E_1$, since $B$ and
$x$ are contained in $E_1$, we can write it as
\begin{equation}
\label{eq02}
\begin{split}
& \sum_{y \in E_1} \mu(y) \, \bb P^1_y[H_{x} < H_B]
\;+\; \sum_{y \not \in  E_1} \sum_{z\in E_1} \mu(y) 
\, q(y, z) \, \bb P^1_z[H_{x} < H_B] \\
&\quad =\; \sum_{y \in E_1} \mu(y) \, \bb P^1_y[H_{x} < H_B]
\;+\; \sum_{z \in E_1} \bb P^1_z[H_{x} < H_B]
\sum_{y \not \in  E_1} \mu(y) \, q(y, z)\;.  
\end{split}
\end{equation}

We claim that for $y\not \in  E_1$, $z\in E_1$,
\begin{equation}
\label{01}
\mu(y) \, q(y, z) \;=\;
\mu(y) \, \bb P_y[H_z < H_{E_1\setminus \{z\}} ] 
\;\prec\; \mu(z)\;.
\end{equation}
If $\mu(y) \prec \mu(z)$, there is nothing to prove. Assume
that $\mu(z) \preceq\mu(y)$. In this case, by \eqref{mf02} and by
\eqref{ehyp1}, the second term in the previous expression is bounded
by
\begin{equation*}
\frac{\mu(y) \, \Cap(y,z)}{\Cap(y , E_1)}
\;\prec\; \mu(z)\;.
\end{equation*}
This proves claim \eqref{01} and that the second term in the last
equation of \eqref{eq02} is of smaller order than the first, as
asserted.
\end{proof}

\begin{remark}
\label{rm1}
In Lemma \ref{mest1}, the set $E_1$ has to be interpreted as the union
of wells. In the set-up of the Metropolis dynamics introduced in
Remark \ref{rm2}, by \eqref{303} and Remark \ref{rm2}, for two
disjoint subsets $A$, $B$ of $E$, $\Cap (A,B)/\mu_N(x_{A,B})$
converges, as $N\uparrow\infty$, to a real number in
$(0,\infty)$. Hence, assumption \eqref{ehyp1} requires that for all
$z\in E_1$, $y \not\in E_1$ such that $\bb H(y)\le \bb H(z)$,
\begin{equation}
\label{304}
\bb H(x_{y,E_1}) \;-\; \bb H(y) 
\;<\; \bb H(x_{y,z}) \;-\; \bb H(z)\;.
\end{equation}
In other words, it requires the energy barrier from $y$ to $E_1$ to be
smaller than the one from $z$ to $y$.
\end{remark}

The condition \eqref{304} may seem unnatural, as one would expect on
the right hand side $\bb H(x_{z, \breve E_z}) - \bb H(z)$ instead of
$\bb H(x_{y,z}) - \bb H(z)$, where $\breve E_z$ represents the union
of the wells which do not contain $z$. However, since in the
applications the set $E_1$ represents the union of wells, and since
$\bb H(y)\le \bb H(z)$, to reach $y$ from $z$ the chain has to jump
from one well to another and therefore one should have $\bb H(x_{z,
  \breve E_z}) - \bb H(z) \le \bb H(x_{y,z}) - \bb H(z)$.

\begin{lemma}
\label{as14}
Fix two points $a \not =b\in E$.  The set of sequences $\mu_N(x) \bb
P_x[H_a<H_b]$, $x\in E\setminus \{b\}$, is ordered.
\end{lemma}

\begin{proof}
Fix two points $x \not =y\in E\setminus \{b\}$. We need to show that
the ratio $\mu_N(x) \bb P_x[H_a<H_b]/\mu_N(y) \bb P_y[H_a<H_b]$ either
converges to some value in $[0,\infty)$, or increases to $\infty$.

Assume that $x\not =a$, $y\not = a$, and consider the trace of the
process $\eta^N(t)$ on $A=\{a,b,x,y\}$. By \cite[Section 6]{bl2}, the
stationary measure of the trace is the measure $\mu_N$ conditioned to
$A$. Denote by $\bb P^A_z$ the distribution of the trace starting from
$z$. It is clear that $\bb P_z[H_a<H_b] = \bb
P^A_z[H_a<H_b]$. Therefore,
\begin{equation*}
\frac{\mu_N(x) \bb P_x[H_a<H_b]}{\mu_N(y) \bb P_y[H_a<H_b]} \;=\;
\frac{\mu^A_N(x) \bb P^A_x[H_a<H_b]}{\mu^A_N(y) \bb P^A_y[H_a<H_b]}\;,
\end{equation*}
where $\mu^A_N$ represents the measure $\mu_N$ conditioned to $A$.
Since $A$ has only four elements, it is not difficult to show that
\begin{equation*}
\bb P^A_x[H_a<H_b] \;=\; \frac{p^A(x,a) + p^A(x,y)p^A(y,a)}{1-p^A(x,y)
  p^A(y,x)}\;, 
\end{equation*}
where $p^A(z,z')$ represents the jump probabilities of the trace
process. In particular, multiplying the numerator and the denominator
of the penultimate ratio by $\lambda^A(x)\lambda^A(y)$, where
$\lambda^A$ stands for the holding rates of the trace process, yields
that the penultimate ratio is equal to
\begin{equation*}
\frac{\mu^A_N(x) \{\lambda^A(y) R^A(x,a) + R^A(x,y)R^A(y,a)\}}
{\mu^A_N(y) \{\lambda^A(x)  R^A(y,a) + R^A(y,x)R^A(x,a)\}}\;,
\end{equation*}
where $R^A$ is the jump rates of the trace process. By reversibility of
the trace process, this expression is equal to
\begin{equation*}
\frac{\lambda^A(y) R^A(a,x) + R^A(y,x)R^A(a,y) }
{\lambda^A(x)  R^A(a,y) + R^A(x,y) R^A(a,x)}\;\cdot
\end{equation*}
By \cite[Lemma 4.3]{bl4}, the set of jump rates $R^A(x,y)$ satisfies
Assumption \ref{mhyp}. Since $\lambda^A(z) = \sum_{z'\in A, z'\not =
  z} R^A(z,z')$, by Assumption \ref{mhyp}, the previous expression
either converges to some $a\in [0,\infty)$, or increases to
$+\infty$. This completes the proof of the assertion in the case where
$x$, $y\not\in \{a,b\}$.

The case where $x=a$ or $y=a$ is simpler and left to the reader.
\end{proof}

\section{Proofs of Propositions \ref{0b1}, \ref{mt4} and \ref{mt3}.A}
\label{sec3}

We examine in this section the metastable behavior of the Blume-Capel
model starting from $\bf -1$. We first consider isovolumetric
inequalities.  Denote by $\Vert \,\cdot\,\Vert$ the Euclidean norm of
$\bb R^2$. A subset $A$ of $\bb Z^2$ is said to be connected if for
every $x$, $y\in A$, there exists a path $\gamma = (x=x_0, x_1, \dots,
x_n=y)$ such that $x_i\in A$, $\Vert x_{i+1} - x_i\Vert =1$, $0\le i<
n$.  Denote by $\mc C_n$, $n\ge 1$, the class of connected subset of
$\bb Z^2$ with $n$ points and by $P(A)$ the perimeter of a set $A\in
\mc C_n$:
\begin{equation*}
P(A) \;=\; \# \{(x,y) \in \bb Z^2 : x\in A \,,\, y\not \in A
\,,\, \Vert x-y\Vert=1\}\;.
\end{equation*}
where $\# B$ stands for the cardinality of $B$.

\begin{asser}
\label{as5}
For every $A\in \mc C_n$, $n\ge 1$, $P(A) \ge 4\sqrt{n}$.
\end{asser}

\begin{proof}
For $A\in \mc C_n$, denote by $R$ the smallest rectangle which
contains $A$, and by $a\le b$ the length of the sides of the rectangle
$R$. Since $A$ is connected, and since $R$ is the smallest rectangle
which contains $A$, $P(A) \ge 2(a+b) \ge 2 \min\{ \mf a+ \mf b : \mf
a, \mf b\in\bb R \,,\, \mf a\mf b \ge n\} = 4\sqrt{n}$.
\end{proof}

\begin{asser}
\label{as6}
A set $A\in \mc C_m$, $m=n_0(n_0+1)$, is either a $n_0 \times (n_0+1)$
rectangle or has perimeter $P(A) \ge 4(n_0+1)$.
\end{asser}

\begin{proof}
Fix $A\in \mc C_m$, and recall the notation introduced in the proof of
the previous assertion. We may restrict our study to the case where
the length of the shortest side of $R$, denoted by $a$, is less than or equal
to $n_0$, otherwise the perimeter is larger than or equal to
$4(n_0+1)$. If $a=n_0$, either $b=n_0+1$, in which case, to match the
volume, $A$ must be a $n_0 \times (n_0+1)$ rectangle, or $b\ge n_0+2$,
in which case the the perimeter is larger than or equal to $4(n_0+1)$.
If $a=n_0-j$ for some $j\ge 1$, then $b=n_0+k$ for some $k\ge 1$
because the volume has to be at least $n^2_0$. Actually, we need
$(n_0+k)(n_0-j) \ge n_0(n_0+1)$, i.e., $(k-j)n_0 \ge n_0 + kj$. This
forces $k-j\ge 2$ and, in consequence, the perimeter $P\ge 4(n_0+1)$.
\end{proof}

We may extend the definition of the energy $\bb H$ introduced in
\eqref{eq06} to configuration in $\{-1,0,1\}^{\bb Z^2}$. For such
configurations, while $\bb H(\sigma)$ is not well defined, $\bb
H(\sigma)-\bb H({\bf -1})$ is well defined if $\sigma_x=-1$ for all
but a finite number of sites.

Denote by $\partial_+ A$ the outer boundary of a connected finite
subset $A$ of $\bb Z^2$: $\partial_+A =\{x\not\in A : \exists\, y\in A
\text{ s.t. } \Vert y-x\Vert = 1\}$. 

\begin{asser}
\label{as4}
Let $A\in \mc C_n$, $1\le n\le (n_0+1)^2$, and let $\sigma$ be a
configuration of $\{-1,0,1\}^{\bb Z^2}$ whose spins in $A$ are equal
to $+1$ and whose spins in $\partial_+A$ are either $0$ or $-1$. Let
$\sigma^\star$ be the configuration obtained from $\sigma$ by
switching all spins in $A$ to $0$. Then, $\bb H(\sigma) \ge \bb
H(\sigma^\star) +2$.
\end{asser}

\begin{proof}
By definition of the energy and since $A$ has $n$ points, $\bb
H(\sigma) - \bb H(\sigma^\star) = - hn + P_0 + 3P_{-1} \ge -hn + P$,
where $P_0$ (resp. $P_{-1}$) represents the number of unordered pairs
$\{x,y\}$ such that $x\in A$, $y\in \partial_+A$, $\sigma_y=0$
(resp. $\sigma_y=-1$), and where $P=P_0 + P_{-1}$ is the perimeter of
the set $A$.

It remains to show that $P-hn\ge 2$. For $1\le n\le 3$, this follows
by inspecting all cases, keeping in mind that $h<1\le 2$. Next, assume
that $n\ge 4$. By hypothesis, and since $n_0\ge 2$, $(2/3) \sqrt{n}
\le (2/3) (n_0+1) \le n_0 < 2/h$ so that $hn < 3\sqrt{n}$. Hence, by
Assertion \ref{as5}, $hn < 4\sqrt{n} - \sqrt{n} \le P-2$.
\end{proof}

Let $A(\sigma) = \{x\in\bb Z^2 : \sigma_x \not = -1\}$, $\sigma\in
\{-1,0,1\}^{\bb Z^2}$. Denote by $\mf B$ the boundary of the valley of
$-1$ formed by the set of configurations with $n_0(n_0+1)$ sites with
spins different from $-1$:
\begin{equation*}
\mf B \;=\; \big\{\sigma\in \{-1,0,1\}^{\bb Z^2} : 
|A(\sigma)| = n_0(n_0+1) \big\}\;.
\end{equation*}
Sometimes, we consider $\mf B$ as a subset of $\Omega$.  Denote by
$\mf R$ the subset of $\mf B$ given by
\begin{equation*}
\mf R \;=\; \big\{ \sigma\in \{-1,0\}^{\bb Z^2} : 
A(\sigma) \text{ is a } n_0\times (n_0+1) \text{ rectangle }\big\} \;.
\end{equation*}
Note that the spins of a configuration $\sigma\in \mf R$ are either
$-1$ or $0$ and that all configurations in $\mf R$ have the same
energy.  

\begin{asser}
\label{lem1}
We have that $\bb H(\sigma) \ge \bb H(\zeta) +2$ for all $\sigma\in
\mf B \setminus \mf R$, $\zeta\in\mf R$.
\end{asser}

\begin{proof}
Fix a configuration $\sigma\in \mf B$. Let $\sigma^*$ be the
configuration of $\{-1,0\}^{\bb Z^2}$ obtained from $\sigma$ by
switching all $+1$ spins to $0$. Applying Assertion \ref{as4} $k$
times, where $k$ is the number of connected components formed by
$+1$ spins, we obtain that $\bb H(\sigma) \ge \bb H(\sigma^*) + 2k$.
It is therefore enough to prove the lemma for configurations
$\sigma\in \{-1,0\}^{\bb Z^2}$.

Let $\sigma$ be a configuration in $\mf B \cap \{-1,0\}^{\bb Z^2}$. If
$A(\sigma)$ is not a connected set, by gluing the connected components
of $A(\sigma)$, we reach a new configuration $\sigma^*\in \mf B \cap
\{-1,0\}^{\bb Z^2}$ such that $A(\sigma^*) \in \mc C_m$,
$m=n_0(n_0+1)$. Since by gluing two components, the volume remains
unchanged, but the perimeter decreases at least by $2$, $\bb H(\sigma)
\ge \bb H(\sigma^*) + 2$. It is therefore enough to prove the lemma
for those configuration in $\mf B \cap \{-1,0\}^{\bb Z^2}$ for which
$A(\sigma)$ is a connected set.

Finally, fix a configuration in $\mf B \cap \{-1,0\}^{\bb Z^2}$ for
which $A(\sigma)$ is a connected set different from a $n_0\times
(n_0+1)$ rectangle.  Since all spins of $\sigma$ are either $-1$ or
$0$. By definition of the energy, $\bb H(\sigma) - \bb H(\zeta) =
P(A(\sigma))- P(A(\zeta))$, and the result follows from Assertion
\ref{as6}.
\end{proof}

Denote by $\mf R^+$ the set of configurations in $\{-1,0, +1\}^{\bb
  Z^2}$ in which there are $n_0(n_0+1)+1$ spins which are not equal to
$-1$. Of these spins, $n_0(n_0+1)$ form a $n_0\times
(n_0+1)$-rectangle of $0$ spins. The remaining spin not equal to $-1$
is either $0$ or $+1$.

It is clear that starting from $\bf -1$ the set $(\mf B \setminus \mf
R) \cup \mf R^+$ is hit before the chain attains the set $\{{\bf 0},
{\bf +1}\}$:
\begin{equation}
\label{eq09}
H_{\mf B^+} \;<\; H_{\{{\bf 0}, {\bf +1}\}} \quad \bb P_{\bf -1} \text{
  a.s.}\;, \quad\text{where $\mf B^+ = (\mf B \setminus \mf R) \cup \mf
  R^+$} \;. 
\end{equation}
Let $\mf R^{a}\subset \mf R^+$ be the set of configurations for which the
remaining spin is a $0$ spin attached to one of the sides of the
rectangle. Note that all configurations of $\mf R^a$ have the same
energy and that $\bb H(\xi) = \bb H(\zeta)+2 -h$ if $\xi\in \mf R^a$,
$\zeta\in \mf R$. In particular, 
\begin{equation}
\label{eq07}
\bb H(\sigma) \;\ge\; \bb H(\xi) \;+\; h \;, \quad 
\sigma\in \mf B \setminus \mf R\;\;,\;\;
\xi\in \mf R^a\;.
\end{equation}
On the other hand, for a configuration $\eta\in\mf R^+ \setminus \mf
R^a$, $\bb H(\eta) \ge \bb H(\zeta)+4 -h$ if $\zeta\in \mf R$, so that
\begin{equation}
\label{eq08}
\bb H(\eta) \;\ge\; \bb H(\xi) \;+\; 2 \;, \quad 
\eta\in \mf R^+ \setminus \mf R^a\;\;,\;\;
\xi\in \mf R^a\;.
\end{equation}
\smallskip

Recall the notation introduced just above Remark \ref{rm2}. For two
disjoint subsets $\ms A$ to $\ms B$ of $\Omega$, denote by $\xi_{\ms
  A, \ms B}$, the configuration with highest energy in the optimal
path joining $\ms A$ to $\ms B$. By Remark \ref{rm2} and \eqref{303},
the limit
\begin{equation}
\label{mf04}
\lim_{\beta\to\infty} \frac{\Cap (\ms A, \ms B)}
{\mu_\beta(\xi_{\ms A, \ms B})} \;\;\text{exists and takes value in}\;\; 
(0,\infty) \;.
\end{equation}
In particular, for every subset $\ms B$ of $\Omega$ and every
configuration $\sigma\not\in\ms B$,
\begin{equation}
\label{401}
\Cap (\sigma, \ms B) \;\preceq\; \frac 1{Z_\beta}\, e^{-\beta \bb H(\sigma)}  \;.
\end{equation}

\begin{asser}
\label{as7}
For all configurations $\xi\in\mf R^a$, $ \Cap (\xi, {\bf -1})
\;\approx\; Z^{-1}_\beta \exp\{-\beta \bb H(\xi)\}$.
\end{asser}

\begin{proof}
Fix $\xi\in \mf R^a$. By \eqref{mf04}, \eqref{401}, it is enough to
exhibit a path $\gamma = (\xi=\xi_0, \xi_1, \dots, \xi_n={\bf -1})$
from $\xi$ to $\bf -1$ such that $\max_i \bb H(\xi_i) = \bb H(\xi)$.

Consider the path $\gamma = (\xi=\xi_0, \xi_1, \dots, \xi_n={\bf
  -1})$, $n=n_0(n_0+1)+1$, constructed as follows. $\xi_1$ is the
configuration obtained from $\xi$ by switching the attached particle
from $0$ to $-1$. Clearly, $\bb H(\xi_1) = \bb H(\xi) -2 +h$, and
$\xi_1$ consists of a $n_0\times (n_0+1)$ rectangle of $0$ spins.

The portion $(\xi_1, \dots, \xi_{n_0+1})$ of the path $\gamma$ is
constructed by flipping, successively, from $0$ to $-1$, all spins of
one of the shortest sides of the rectangle, keeping until the last
step the perimeter of the set $A(\xi_i)$ equal to $4n_0(n_0+1)$. In
particular, $\xi_{n_0+1}$ consists of a $n_0\times n_0$ square of $0$
spins, $\bb H(\xi_{i+1}) = \bb H(\xi_{i}) +h$ for $1\le i < n_0$, and
$\bb H(\xi_{n_0+1}) = \bb H(\xi_{n_0}) -2+h$. The energy of this piece
of the path attains its maximum at $\xi_{n_0}$ and $\bb H(\xi_{n_0}) =
\bb H(\xi_{1}) + (n_0-1)h = \bb H(\xi) -2 + n_0h < \bb H(\xi)$.

The path proceed in this way by always flipping from $0$ to
$-1$ all spins of one of the shortest sides. It is easy to check that
$\bb H(\xi_i) < \bb H(\xi)$ for all $1\le i\le n$, proving the
assertion. 
\end{proof}

\begin{asser}
\label{as3}
For every $\sigma\in \mf B^+\setminus \mf R^a$,
$\lim_{\beta\to\infty} \bb P_ {\bf -1} [H_{\sigma} = H_{\mf B^+} ] =0$.
\end{asser}

\begin{proof}
Fix $\sigma\in \mf B^+ \setminus \mf R^a$ and $\xi\in\mf R^a$. By
\eqref{mf02}; by the monotonicity of the capacity, stated in
\cite[Lemma 2.2]{gl14}, and by \eqref{401}; and by Assertion
\ref{as7},
\begin{equation*}
\bb P_{{\bf -1}} [H_{\sigma} = H_{\mf B^+}]  \;\le\;
\frac{ \Cap(\sigma , {\bf -1})} {\Cap(\mf B^+ , {\bf -1})} 
\;\le\;
\frac{C_0}{Z_\beta} \, \frac{ e^{-\beta \bb H(\sigma)} } {\Cap(\xi , {\bf -1})} 
\;\le \; C_0\, e^{-\beta \{\bb H(\sigma) - \bb H(\xi)\}} 
\end{equation*}
for some finite constant $C_0$ independent of $\beta$. By \eqref{eq07},
\eqref{eq08}, this expression vanishes as $\beta\uparrow\infty$, proving
the assertion.
\end{proof}

Denote by $\ms S$ the set of stable configurations:
\begin{equation}
\label{stable}
\ms S \;=\; \big\{\sigma \in\Omega : \lim_{\beta\to\infty}
\lambda_\beta(\sigma)=0\big\}\;. 
\end{equation}
The next assertion is the only one in which capacities are not used to
derive the needed bounds, because we estimate the probability of
reaching a state which can be attained through paths in which the
energy never increase. The argument, though, is fairly simple.

\begin{figure}
  \centering
\begin{tikzpicture}[scale = .3]
\foreach \y in {0, ..., 4}
\foreach \x in {0, ..., 5}
\draw[very thick] (\x, \y) -- (\x +1, \y) -- (\x
+ 1, \y + 1) -- (\x , \y + 1) -- (\x , \y );
\draw[very thick] (-1, 4) -- (0, 4) -- (0, 5) -- (-1 , 5) -- 
(-1 , 4);
\foreach \x in {9, ..., 14}
\foreach \y in {0, ..., 4}
\draw[very thick] (\x, \y) -- (\x +1, \y) -- (\x
+ 1, \y + 1) -- (\x , \y + 1) -- (\x , \y );
\draw[very thick] (10, 5) -- (11, 5) -- (11,6) -- (10,6) -- (10,5);
\end{tikzpicture}
\caption{Examples of configurations $\sigma\in\mf R^c$ and
  $\sigma'\in\mf R^i$ in the case where $n_0=5$. A $1\times 1$ square
  centered at $x$ has been placed at each site $x$ occupied by a
  $0$-spin. All the other spins are equal to $-1$.}
\label{fig1}
\end{figure}
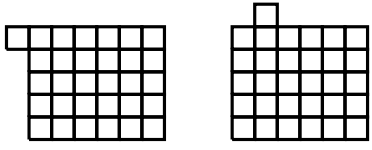

Denote by $\mf R^{c}$, $\mf R^{i}$ the configurations of $\mf R^{a}$
in which the extra particle is attached to the corner, interior of the
rectangle, respectively (cf. Figure \ref{fig1}).

\begin{asser}
\label{as9}
For $\sigma\in \mf R^{c}$ and $\sigma'\in \mf R^{i}$,
\begin{align*}
& \lim_{\beta\to\infty} 
\bb P_{\sigma} [H_{\sigma_+} = H_{\ms S}]  \;=\; 1/2
\quad \text{and}\quad
\lim_{\beta\to\infty} 
\bb P_{\sigma} [H_{\sigma_-} = H_{\ms S}]  \;=\; 1/2\;, \\
& \quad \lim_{\beta\to\infty} 
\bb P_{\sigma'} [H_{\sigma_+} = H_{\ms S}]  \;=\; 2/3
\quad \text{and}\quad
\lim_{\beta\to\infty} 
\bb P_{\sigma'} [H_{\sigma_-} = H_{\ms S}]  \;=\; 1/3\;,
\end{align*}
where $\sigma_-$ is the configuration obtained from $\sigma$ or
$\sigma'$ by flipping to $-1$ the attached $0$ spin, and $\sigma_+$ is
the configuration whose set $A(\sigma_+)$, formed only by $0$ spins,
is the smallest rectangle which contains $A(\sigma)$.
\end{asser}

\begin{proof}
Suppose that the extra $0$ spin is not attached to the corner of the
rectangle. Denote by $\sigma_1$, $\sigma_2$ the configurations
obtained from $\sigma'$ by flipping from $-1$ to $0$ one of the two
$-1$ spins which has two neighbor spins equal to $0$, and let
$\sigma_0=\sigma_-$. By definition, $R_\beta(\sigma',\sigma_j) =1$,
$0\le j\le 2$, and $R_\beta(\sigma',\sigma')=o(1)$ for all the other
configurations, where $o(1)$ represents an expression
which vanishes as $\beta\uparrow\infty$.  This shows that
$p_\beta(\sigma',\sigma_j)$ converges to $1/3$, $0\le j\le 2$.  We may
repeat this argument to show that from $\sigma_j$, $j=1$, $2$, one
reaches $\ms S$ at $\sigma_+$ with a probability asymptotically equal
to $1$. The argument is similar if the extra spin is attached to the
corner.
\end{proof}

\begin{asser}
\label{as10}
Fix a configuration $\sigma\in\ms S$ for which $A(\sigma)$ is a
$m\times n$ rectangle of $0$ spins in a sea of $-1$ spins. Assume that
$m\le n$. Then, 
\begin{equation*}
\lim_{\beta\to\infty} 
\bb P_{\sigma} [H_{\ms B} = H_{\ms S\setminus\{\sigma\}}]  \;=\; 1\;.
\end{equation*}
In this equation, If $n_0<m$, $n\le L-3$, $\ms B$ is the set of four
configurations in which a row or a column of $0$ spins is added to the
rectangle $A(\sigma)$. If $n_0<m < n=L-2$, the set $\ms B$ is a triple
which includes a band of $0$ spins of width $m$ and two configurations
in which a row or a column of $0$ spins of length $n$ is added to the
rectangle $A(\sigma)$.  If $n_0<m\le L-3$, $n=L$, the set $\ms B$ is a
pair formed by two bands of $0$ spins of width $m+1$. If $n_0<m=
n=L-2$, $\ms B$ is a pair of two bands of width $L-2$. If $n_0<m=L-2$,
$n=L$, $\ms B =\{{\bf 0}\}$. Finally, if $2\le m\le n_0$, $n\ge 3$,
the set $\ms B$ is the pair (quaternion if $m=n$) of configurations in
which a row or a column of $0$ spins of length $m$ is removed from the
rectangle $A(\sigma)$, and if $m=n=2$, $\ms B =\{{\bf -1}\}$.
\end{asser}

\begin{proof}
The assertion follows from inequality \eqref{mf02}, and from estimates
of $\Cap (\sigma, \ms S\setminus\{\sigma\})$, $\Cap (\sigma, \ms
S\setminus [\ms B \cup\{\sigma\}])$. In the case $m>n_0$, $\Cap
(\sigma, \ms S\setminus\{\sigma\}) \approx \mu_\beta(\sigma)
\exp\{-\beta (2-h)\}$ and $\Cap (\sigma, \ms S\setminus [\ms B
\cup\{\sigma\}]) \prec \mu_\beta(\sigma) \exp\{-\beta (2-h)\}$, while
in the case $m\le n_0$, $\Cap (\sigma, \ms S\setminus\{\sigma\})
\approx \mu_\beta(\sigma) \exp\{-\beta (m-1)h\}$ and $\Cap (\sigma,
\ms S\setminus [\ms B \cup\{\sigma\}]) \prec \mu_\beta(\sigma)
\exp\{-\beta (m-1)h\}$.
\end{proof}

\begin{lemma}
\label{as61}
For every $\sigma\in \mf B^+$, 
\begin{equation*}
\lim_{\beta\to\infty} \bb P_ {\bf -1}
[H_{\sigma} = H_{\mf B^+} ] \;=\; \frac 1{|\mf R^a|} 
\, \mb 1\{\sigma\in\mf R^a\}\;,
\end{equation*}
where $|\mf R^a|$ represents the number of configurations in $\mf
R^a$.
\end{lemma}

\begin{proof}
In view of Assertion \ref{as3}, we may restrict our attention to 
$\sigma\in \mf R^a$. Fix a reference configuration $\sigma^*$ in
$\mf R^a$. By \eqref{mf01} and by definition of the
capacity, 
\begin{equation*}
\bb P_ {\bf -1} [H_{\sigma} = H_{\mf B^+} ] \;=\; \frac{M({\bf -1}) 
\, \bb P_ {\bf -1} [H_{\sigma} = H^+_{\mf B^+ \cup\{{\bf -1}\}} ]}
{\Cap ({\bf -1}, \mf B^+)}\;\cdot
\end{equation*}
By reversibility, the numerator of this expression is equal to
\begin{equation*}
M(\sigma) \, \bb P_ {\sigma} [H_{\bf -1} = H^+_{\mf B^+ \cup\{{\bf -1}\}} ]
\;=\; \mu_\beta(\sigma) \, \lambda(\sigma) \, 
\bb P_ {\sigma} [H_{\bf -1} = H^+_{\mf B^+ \cup\{{\bf -1}\}} ]\;.
\end{equation*}
By Assertions \ref{as9} and \ref{as10}, $\bb P_ {\sigma} [H_{\bf -1} =
H^+_{\mf B^+ \cup\{{\bf -1}\}} ] = \mf n (\sigma) +o(1)$, where
\begin{equation*}
\mf n(\sigma) \;=\;
\begin{cases}
1/2  & \text{ if $\sigma\in\mf R^{c}$,} \\
1/3 & \text{ if $\sigma\in\mf R^{i}$.}
\end{cases}
\quad\text{Since}\quad
\lambda(\sigma) \;=\;
\begin{cases}
2 + o(1) & \text{ if $\sigma\in\mf R^{c}$,} \\
3 +o(1) & \text{ if $\sigma\in\mf R^{i}$.}
\end{cases}
\end{equation*}
Since $\mu_\beta(\sigma) = \mu_\beta(\sigma^*)$, we conclude that
\begin{equation*}
\bb P_ {\bf -1} [H_{\sigma} = H_{\mf B^+} ] \;=\; 
\frac{\mu_\beta(\sigma^*)}
{\Cap ({\bf -1}, \mf B^+)} \, \Big( \mb 1\{\sigma \in \mf R^a\} 
+ o(1) \Big)\;.
\end{equation*}
Summing over $\sigma\in\mf B^+$, we conclude that
$\mu_\beta(\sigma^*)/\Cap ({\bf -1}, \mf B^+) = |\mf R^a|^{-1}
(1+o(1))$, which completes the proof of the assertion.
\end{proof}

It follows from the proof of the previous lemma that for any
configuration $\sigma^*\in\mf R^a$,
\begin{equation}
\label{410}
\lim_{\beta\to\infty}
\frac{\Cap ({\bf -1}, \mf B^+)}{\mu_\beta(\sigma^*)} 
\;=\; |\mf R^a|\;.
\end{equation}

Denote by $\mf R^{l}$, $\mf R^{s}$ the configurations of $\mf R^a$ in
which the extra particle is attached to one of the longest, shortest
sides, respectively, and let $\mf R^{lc} = \mf R^{l} \cap \ms R^{c}$,
$\mf R^{li}= \mf R^{l} \cap \mf R^{i}$.  The next lemma is an
immediate consequence of Assertions \ref{as9} and \ref{as10}.

\begin{lemma}
\label{as11b}
For $\sigma\in \mf R^{lc}$, $\sigma'\in \mf R^{li}$, and $\sigma''\in \mf R^{s}$,
\begin{align*}
& \lim_{\beta\to\infty} 
\bb P_{\sigma} [H_{\bf -1} = H_{\ms M}]  \;=\; 1/2
\quad \text{and}\quad
\lim_{\beta\to\infty} 
\bb P_{\sigma} [H_{\bf 0} = H_{\ms M}]  \;=\; 1/2\;, \\
& \quad \lim_{\beta\to\infty} 
\bb P_{\sigma'} [H_{\bf -1} = H_{\ms M}]  \;=\; 1/3
\quad \text{and}\quad
\lim_{\beta\to\infty} 
\bb P_{\sigma'} [H_{\bf 0} = H_{\ms M}]  \;=\; 2/3\;, \\
& \qquad \lim_{\beta\to\infty} 
\bb P_{\sigma''} [H_{\bf -1} = H_{\ms M}]  \;=\; 1\;.
\end{align*}
\end{lemma}

It follows from this assertion that for every $\sigma\in \mf R^a$, 
\begin{equation}
\label{411}
\lim_{\beta\to\infty} 
\bb P_{\sigma} [H_{\{{\bf -1} ,\bf{0}\}} < H_{\bf +1}]  \;=\; 1\;,
\end{equation}

\begin{proof}[Proof of Proposition \ref{mt3}, Part A]
We first claim that
\begin{equation}
\label{601}
\Cap ({\bf -1} , \{{\bf 0}, {\bf +1}\}) \;=\;
\Cap ({\bf -1} , \mf B^+)
\sum_{\sigma\in\mf B^+} \bb P_{\bf -1} [ H_{\sigma} = H_{\mf B^+} ]  \,
\bb P_{\sigma} [ H_{\{{\bf 0}, {\bf +1}\} }  < H_{\bf -1} ] \;.
\end{equation}
Indeed, since starting from $\bf -1$ the process hits $\mf B^+$ before
$\{{\bf 0}, {\bf +1}\}$, by the strong Markov property we have that
\begin{equation*}
\bb P_{\bf -1} [ H_{\{{\bf 0}, {\bf +1}\} }  < H^+_{\bf -1} ] \;=\;
\sum_{\sigma\in\mf B^+} \bb P_{\bf -1} [ H_{\sigma}
= H^+_{\mf B^+ \cup \{\bf -1\}} ]  \,
\bb P_{\sigma} [ H_{\{{\bf 0}, {\bf +1}\} }  < H_{\bf -1} ] \;.
\end{equation*}
By \eqref{mf01}, we may rewrite the previous expression as
\begin{equation*}
\bb P_{\bf -1} [ H_{\mf B^+} < H^+_{\bf -1} ]  
\sum_{\sigma\in\mf B^+} \bb P_{\bf -1} [ H_{\sigma} = H_{\mf B^+} ]  \,
\bb P_{\sigma} [ H_{\{{\bf 0}, {\bf +1}\} }  < H_{\bf -1} ] \;.
\end{equation*}
This proves \eqref{601} in view of the definition \eqref{202} of the
capacity. 

By \eqref{601} and \eqref{410}, for any configuration $\sigma^*\in\mf
R^a$,
\begin{equation*}
\lim_{\beta\to\infty} \frac{\Cap ({\bf -1} , \{{\bf 0}, {\bf +1}\})} 
{\mu_\beta(\sigma^*)} \;=\; |\mf R^a|\, \lim_{\beta\to\infty} 
\sum_{\sigma\in\mf B^+} \bb P_{\bf -1} [ H_{\sigma} = H_{\mf B^+} ]  \,
\bb P_{\sigma} [ H_{\{{\bf 0}, {\bf +1}\} }  < H_{\bf -1} ] \;.
\end{equation*}
By Lemma \ref{as61}, the right hand side is equal to
\begin{equation}
\label{415}
\lim_{\beta\to\infty} 
\sum_{\sigma\in\mf R^a} 
\bb P_{\sigma} [ H_{\{{\bf 0}, {\bf +1}\} }  < H_{\bf -1} ] \;.
\end{equation}
By Lemma \ref{as11b}, this expression is equal to
$(1/2) |\mf R^{lc}| + (2/3) |\mf R^{li}| =
2|\Lambda_L|\{2+(4/3)(n_0-1)\}$, which completes the proof of the
first claim of the proposition.
\end{proof}

\begin{asser}
\label{as18}
We have that
\begin{equation*}
\lim_{\beta\to\infty} 
\frac{\Cap ({\bf -1}, {\bf 0})}{\Cap ({\bf -1}, \{{\bf 0} , 
{\bf  +1}\})}\;=\; 1 \;.
\end{equation*}
\end{asser}

\begin{proof}
Let $\sigma^*$ be a configuration in $\mf R^a$.
By the proof of Proposition \ref{mt3} up to \eqref{415},
\begin{equation*}
\lim_{\beta\to\infty} 
\frac{\Cap ({\bf -1}, {\bf 0})}{\mu_\beta(\sigma^*)}\;=\; 
\lim_{\beta\to\infty} 
\sum_{\sigma\in\mf R^a} \bb P_{\sigma} [ H_{\bf 0}  < H_{\bf -1} ] \;.
\end{equation*}
By \eqref{411}, this expression is equal to \eqref{415}. This
completes the proof of the assertion.
\end{proof}

\begin{proof}[Proof of Proposition \ref{0b1}]
Let $q(\sigma) = \bb P_ {\bf -1} [H_{\sigma} = H_{\mf B^+} ]$,
$\sigma\in\mf B^+$. By Assertion \ref{as3}, $q(\sigma)\to 0$ if
$\sigma\not\in\mf R^a$. Hence, by \eqref{eq09},
\begin{equation*}
\bb P_ {\bf -1} [H_{\bf +1} < H_{\bf 0} ]\;=\; \sum_{\sigma\in\mf B^+}
q(\sigma) \, \bb P_ {\sigma} [H_{\bf +1} < H_{\bf 0} ]
\;=\; \sum_{\sigma\in\mf R^a}
q(\sigma) \, \bb P_ {\sigma} [H_{\bf +1} < H_{\bf 0} ] \;+\; o(1) \;,
\end{equation*}
By \eqref{411}, for all $\sigma\in\mf R^a$, 
\begin{equation*}
\bb P_ {\sigma} [H_{\bf +1} < H_{\bf 0} ] \;=\;
\bb P_ {\sigma} [H_{\bf +1} < H_{\bf 0} \,,\,
H_{\{{\bf -1},{\bf 0}\}} < H_{\bf +1} ] \;=\;
\bb P_ {\sigma} [ H_{\bf -1} < H_{\bf +1} < H_{\bf 0}]\;.
\end{equation*}
Therefore, by the strong Markov property, 
\begin{equation*}
\bb P_ {\bf -1} [H_{\bf +1} < H_{\bf 0} ]\;=\;
\bb P_ {\bf -1} [H_{\bf +1} < H_{\bf 0} ] 
\sum_{\sigma\in\mf R^a} q(\sigma) \, \bb P_ {\sigma} [H_{\bf -1} <
H_{\{{\bf 0},{\bf +1}\}} ]\, \;+\; o(1) \;.
\end{equation*}
By Lemma \ref{as11b}, for $\sigma\in \mf R^l$,
$\limsup_{\beta\to\infty}\bb P_ {\sigma} [H_{\bf -1} < H_{\{{\bf
    0},{\bf +1}\}} ] \le 1/2$, which completes the proof of the
proposition.
\end{proof}

\begin{proof}[Proof of Proposition \ref{mt4}]
We prove the proposition when the chain starts from $\bf -1$, the
argument being analogous when it starts from $\bf 0$. Since the chains
hits $\mf B^+$ before reaching $\bf 0$ and $\mf R^l$, by the strong
Markov property,
\begin{equation*}
\bb P_{\bf -1} [ H_{\mf R^l} < H_{\bf 0}] \;=\;
\sum_{\sigma\in\mf B^+} \bb P_{\bf -1} [ H_{\sigma} = H_{\mf B^+} ] \,
\bb P_{\sigma} [ H_{\mf R^l} < H_{\bf 0}]\;. 
\end{equation*}
By Lemma \ref{as61}, this expression is equal to
\begin{equation*}
\big (1+ o(1)\big) \frac 1{|\mf R^a|} \Big\{ |\mf R^l| \,+\,
\sum_{\sigma\in\mf R^s} \bb P_{\sigma} [ H_{\mf R^l} < H_{\bf 0}] \Big\}\;. 
\end{equation*} 
By Assertions \ref{as9} and \ref{as10}, for all $\sigma\in\mf R^s$,
$\sigma'\in\mf R$ 
\begin{equation*}
\lim_{\beta\to\infty} \bb P_{\sigma} [ H_{\mf R} < H_{\mf R^l
  \cup\{{\bf -1}, {\bf 0}\}}]\;=\; 1\;, \quad
\lim_{\beta\to\infty} \bb P_{\sigma'} [ H_{\bf -1} < H_{\mf R^l
  \cup\{ {\bf 0}\}}]\;=\; 1 \;. 
\end{equation*} 
Therefore, for all $\sigma\in\mf R^s$,
\begin{equation*}
\lim_{\beta\to\infty} \bb P_{\sigma} [ H_{\bf -1} < H_{\mf R^l 
\cup\{ {\bf 0}\}}]\;=\; 1\;. 
\end{equation*} 
Hence, by the strong Markov property and by the first two identities
of this proof,
\begin{equation*}
\bb P_{\bf -1} [ H_{\mf R^l} < H_{\bf 0}] \;=\;
\big (1+ o(1)\big) \frac 1{|\mf R^a|} \Big\{ |\mf R^l| \,+\,
\sum_{\sigma\in\mf R^s} \bb P_{\bf -1} [ H_{\mf R^l} < H_{\bf 0}] \Big\}\;,
\end{equation*}
which completes the proof of the proposition.  
\end{proof}

\section{Proofs of Propositions \ref{mt3}.B and \ref{mt2}}
\label{sec4}

We examine in this section the metastable behavior of the Blume-Capel
model starting from $\bf 0$. The main observation is that the energy
barrier from $\bf 0$ to $\bf -1$ is larger that the one from $\bf 0$
to $\bf +1$. We may therefore ignore $\bf -1$ and argue by symmetry
that the passage from $\bf 0$ to $\bf +1$ is identical to the one from
$\bf -1$ to $\bf 0$.

In analogy to the notation introduced right before Assertion \ref{lem1},
let $\mf B_0$ be the set of configurations with $n_0(n_0+1)$ sites
with spins different from $0$, and let $\mf R_0$ be the subset of $\mf
B_0$ given by
\begin{equation*}
\mf R_0 \;=\; \big\{ \sigma\in \{0, +1\}^{\bb Z^2} : 
\{x : \sigma(x) \not = 0\} \text{ forms a } n_0\times (n_0+1) 
\text{ rectangle }\big\} \;.
\end{equation*}
Denote by $\mf R^+_0$ the set of configurations in $\{-1,0, +1\}^{\bb
  Z^2}$ in which there are $n_0(n_0+1)+1$ spins which are not equal to
$0$, and in which $n_0(n_0+1)$ spins of magnetization $+1$ form a
$n_0\times (n_0+1)$-rectangle. Finally, let $\mf R^a_0\subset \mf R^+_0$
be the set of configurations for which the remaining spin is a $+1$
spin attached to one of the sides of the rectangle, and let
\begin{equation*}
\mf B^+_0 \;=\; (\mf B_0 \setminus \mf R_0) \,\cup\, 
\mf  R^+_0 \;. 
\end{equation*}

\begin{asser}
\label{as62}
For every $\sigma\in \mf B^+_0$, $\lim_{\beta\to\infty} \bb P_ {\bf 0}
[H_{\sigma} = H_{\mf B^+_0} ] = |\mf R^a_0|^{-1} \mb 1\{\sigma\in\mf
R^a_0\}$. Moreover, if $\sigma^*$ represents a configuration in $\mf
R^a_0$, 
\begin{equation*}
\lim_{\beta\to\infty}
\frac{\Cap ({\bf 0}, \mf B^+_0)}{\mu_\beta(\sigma^*)} 
\;=\; |\mf R^a_0|\;.
\end{equation*}
\end{asser}

\begin{proof}
As in Assertion \ref{as3}, we may exclude all configurations
$\sigma\in \{0,+1\}^{\Lambda_L}$ which do not belong to $\mf R^a_0$. 
We may also exclude all configurations in $\mf B^+_0$ which have a
negative spin since by turning all negative spins into positive spins
we obtain a new configurations whose energy is strictly smaller than
the one of the original configuration. For the configurations in $\mf
R^a_0$ we may apply the arguments presented in the proof of Lemma
\ref{as61}. 
\end{proof}

Denote by $\mf R^{c}_0$, $\mf R^{i}_0$ the configurations of $\mf
R^{a}_0$ in which the extra particle is attached to the corner,
interior of the rectangle, respectively.  Denote by $\mf R^{l}_0$,
$\mf R^{s}_0$ the configurations of $\mf R^a_0$ in which the extra
particle is attached to one of the longest, shortest sides,
respectively, and let $\mf R^{lc}_0 = \mf R^{l}_0 \cap \ms R^{c}_0$,
$\mf R^{li}_0= \mf R^{l}_0 \cap \mf R^{i}_0$.  The proof of the next
assertion is analogous to the one of Lemma \ref{as11b} since it
concerns configurations with only $0$ and $+1$ spins.

\begin{asser}
\label{as63}
For $\sigma\in \mf R^{lc}_0$, $\sigma'\in \mf R^{li}_0$, and
$\sigma''\in \mf R^{s}_0$, 
\begin{align*}
& \lim_{\beta\to\infty} 
\bb P_{\sigma} [H_{\bf 0} = H_{\ms M}]  \;=\; 1/2
\quad \text{and}\quad
\lim_{\beta\to\infty} 
\bb P_{\sigma} [H_{\bf +1} = H_{\ms M}]  \;=\; 1/2\;, \\
& \quad \lim_{\beta\to\infty} 
\bb P_{\sigma'} [H_{\bf 0} = H_{\ms M}]  \;=\; 1/3
\quad \text{and}\quad
\lim_{\beta\to\infty} 
\bb P_{\sigma'} [H_{\bf +1} = H_{\ms M}]  \;=\; 2/3\;, \\
& \qquad \lim_{\beta\to\infty} 
\bb P_{\sigma''} [H_{\bf 0} = H_{\ms M}]  \;=\; 1\;.
\end{align*}
\end{asser}

The next claim follows from the previous two assertions.

\begin{asser}
\label{as65}
We have that
\begin{equation*}
\lim_{\beta\to\infty} \bb P_{\bf 0} [H_{\bf -1} < H_{\bf +1}]
\;=\; 0 \;.
\end{equation*}
\end{asser}

\begin{proof}
Since, starting from $\bf 0$, the set $\mf B^+_0$ is reached before
the process hits $\{{\bf - 1}, {\bf +1}\}$, by the strong Markov
property, 
\begin{equation*}
\lim_{\beta\to\infty} \bb P_{\bf 0} [H_{\bf -1} < H_{\bf +1}]
\;=\; \lim_{\beta\to\infty} \sum_{\sigma\in \mf B^+_0}
\bb P_{\bf 0} [H_{\sigma} = H_{\mf B^+_0}] \, \bb P_{\sigma} [H_{\bf -1}
< H_{\bf +1}]\;. 
\end{equation*}
By Assertion \ref{as62} and by the strong Markov property at time
$H_{\ms M}$, this expression is equal to
\begin{equation*}
\lim_{\beta\to\infty} \frac 1{|\mf R^a_0|} \sum_{\sigma\in \mf R^a_0}
\bb E_{\sigma} \Big[ \bb P_{\sigma(H_{\ms M})} [H_{\bf -1} < H_{\bf
  +1}]\,  \Big] 
\; =\;  c_0 \lim_{\beta\to\infty} 
\bb P_{\bf 0 } [H_{\bf -1} < H_{\bf +1}] \;. 
\end{equation*}
where we applied Assertion \ref{as63} to derive the last identity. In
this equation, $c_0 = \{|\mf R^{lc}_0|/2|\mf R^a_0| \} + \{|\mf
R^{li}_0|/3|\mf R^a_0|\} <1$. This completes the proof of the assertion.
\end{proof}

\begin{proof}[Proof of Proposition \ref{mt3}, Part B]
The proof is similar to the one of Part A, presented in the previous
section. As in \eqref{601}, we have that
\begin{equation*}
\Cap ({\bf 0} , \{{\bf -1}, {\bf +1}\}) \;=\;
\Cap ({\bf 0} , \mf B^+_0)
\sum_{\sigma\in\mf B^+_0} \bb P_{\bf 0} [ H_{\sigma} = H_{\mf B^+_0} ]  \,
\bb P_{\sigma} [ H_{\{{\bf -1}, {\bf +1}\} }  < H_{\bf 0} ] \;.
\end{equation*}
Hence, by Assertion \ref{as62}, for any configuration $\sigma^*\in\mf
R^a_0$,
\begin{equation*}
\lim_{\beta\to\infty} \frac{\Cap ({\bf 0} , \{{\bf -1}, {\bf +1}\})} 
{\mu_\beta(\sigma^*)} \;=\; |\mf R^a_0|\, \lim_{\beta\to\infty} 
\sum_{\sigma\in\mf B^+_0} \bb P_{\bf 0} [ H_{\sigma} = H_{\mf B^+_0} ]  \,
\bb P_{\sigma} [ H_{\{{\bf -1}, {\bf +1}\} }  < H_{\bf 0} ] \;.
\end{equation*}
By Assertion \ref{as62}, the right hand side is equal to
\begin{equation*}
\lim_{\beta\to\infty} 
\sum_{\sigma\in\mf R^a_0} 
\bb P_{\sigma} [ H_{\{{\bf -1}, {\bf +1}\} }  < H_{\bf 0} ] \;.
\end{equation*}
By Assertion \ref{as63}, this expression is equal to
$(1/2) |\mf R^{lc}_0| + (2/3) |\mf R^{li}_0| =
2|\Lambda_L|\{2+(4/3)(n_0-1)\}$, which completes the proof of the
second claim of the proposition.
\end{proof}

As in Assertion \ref{as18} we have that 
\begin{equation}
\label{603}
\lim_{\beta\to\infty} 
\frac{\Cap ({\bf 0}, {\bf +1})}{\Cap ({\bf 0}, \{{\bf -1} , 
{\bf  +1}\})}\;=\; 1 \;.
\end{equation}

\begin{asser}
\label{as64}
We have that
\begin{equation*}
\lim_{\beta\to\infty} 
\frac{\Cap ({\bf -1}, {\bf +1})}{\Cap ({\bf -1}, \{{\bf 0} , 
{\bf  +1}\})}\;=\; 1 \;.
\end{equation*}
\end{asser}

\begin{proof}
We repeat the proof of the part A of Proposition \ref{mt3} up
\eqref{415} to obtain that
\begin{equation*}
\lim_{\beta\to\infty} 
\frac{\Cap ({\bf -1}, {\bf +1})}{\mu_\beta(\sigma^*)}\;=\; 
\lim_{\beta\to\infty} 
\sum_{\sigma\in\mf R^a} \bb P_{\sigma} [ H_{\bf +1}  < H_{\bf -1} ] \;,
\end{equation*}
if $\sigma^*$ represents a configuration in $\mf R^a$. By \eqref{411}, this
expression is equal to
\begin{align*}
& \lim_{\beta\to\infty} \sum_{\sigma\in\mf R^a} \bb P_{\sigma} 
[ H_{\bf 0}  <  H_{\bf +1}  < H_{\bf -1} ] \\
& \qquad \;=\; \lim_{\beta\to\infty} \bb P_{\bf 0} [ H_{\bf +1}  < H_{\bf -1} ]
\sum_{\sigma\in\mf R^a} \bb P_{\sigma} 
[ H_{\bf 0}  <  H_{\{\bf -1 , \bf +1\}} ]\;.
\end{align*}
where we used the strong Markov property in the last step. By
Lemma \ref{as11b} and Assertion \ref{as65}, this limit is equal
to $(1/2) |\mf R^{lc}| + (2/3) |\mf R^{li}|$, which completes the
proof of the assertion.
\end{proof}

\begin{asser}
\label{500}
We have that
\begin{equation*}
\lim_{\beta\to\infty} \frac{\Cap({\bf +1}, \{{\bf -1}, {\bf 0}\})} 
{\Cap({\bf 0}, \{{\bf -1}, {\bf +1}\})} \;=\; 1\;.
\end{equation*}
\end{asser}

\begin{proof}
Indeed, by monotonicity of the capacity and by \eqref{mf06},
\begin{equation*}
\Cap({\bf +1}, {\bf 0}) \;\le\;
\Cap({\bf +1}, \{{\bf -1}, {\bf 0}\}) \;\le\;
\Cap({\bf +1}, {\bf 0}) \;+\; \Cap({\bf +1}, {\bf -1})\;.
\end{equation*}
By Assertion \ref{as64}, by \eqref{603}, and by Proposition \ref{mt3},
$\Cap({\bf +1}, {\bf -1})/ \Cap({\bf 0}, {\bf +1}) \to 0$ as
$\beta\uparrow\infty$. Hence,
\begin{equation*}
\lim_{\beta\to\infty} \frac{\Cap({\bf +1}, \{{\bf -1}, {\bf 0}\})} 
{\Cap({\bf 0}, {\bf +1})} \;=\; 1\;.
\end{equation*}
To complete the proof, it remains to recall \eqref{603}.
\end{proof}

We turn to the proof of Proposition \ref{mt2}.  We first show that the
assumption of Lemma \ref{mest1} are in force. Recall Remark \ref{rm1}.

\begin{asser}
\label{as15}
Consider two configurations $\sigma\not\in \ms M$ and $\eta\in\ms
M$. If $\bb H(\sigma) \le \bb H(\eta)$, then $\bb H(\xi_{\sigma, \ms
  M}) - \bb H(\sigma) < \bb H(\xi_{\sigma,\eta}) - \bb H(\eta)$. 
\end{asser}

\begin{proof}
We claim that for any configuration $\sigma\not\in \ms M$, $\bb
H(\xi_{\sigma, \ms M}) - \bb H(\sigma)\le 2-h$. To prove this claim it
is enough to exhibit a self-avoiding path from $\sigma$ to $\ms M$
whose energy is kept below $\bb H(\sigma) + 2-h$. This is
easy. Starting from $\sigma$ we may first reach the set $\ms S$ of
stable configurations through a path whose energy does not
increase. Denote by $\sigma^\star$ the configuration in $\ms S$
attained through this path. From $\sigma^\star$ we may reach the set
$\ms M$ by removing all small droplets (the ones whose smaller side
has length $n_0$ or less) and by increasing the large droplets (the
ones whose both sides have length at least $n_0+1$) in such a way that
the energy remains less than or equal to $\bb H(\sigma^\star) + 2-h$.
This proves the claim.

On the other hand, since $\bb H(\zeta) \ge \bb H(\eta) +4 -h$ for any
configuration $\zeta$ which differs from $\eta$ at one site, $\bb
H(\xi_{\sigma,\eta}) - \bb H(\eta) \ge 4 -h$, which proves the assertion.
\end{proof}

\begin{asser}
\label{as66}
We have that
\begin{equation*}
\lim_{\beta\to\infty} \frac{M ({\bf -1}) \, \bb P_{\bf -1} 
[H_{\bf 0} < H^+_{\{{\bf -1}, {\bf +1}\}}]}
{\Cap ({\bf -1} , \{{\bf 0}, {\bf +1}\}) }\;=\; 1 \;.
\end{equation*}
\end{asser}

\begin{proof}
Fix $\sigma^*$ in $\mf R^a$. In view of Proposition \ref{mt3}, it is
enough to show that
\begin{equation*}
\lim_{\beta\to\infty} \frac{M ({\bf -1}) \, \bb P_{\bf -1} 
[H_{\bf 0} < H^+_{\{{\bf -1}, {\bf +1}\}}]}
{\mu_\beta(\sigma^*)}\;=\; \frac{4(2n_0+1)}3\, |\Lambda_L|\;. 
\end{equation*}
In the proof of Proposition \ref{mt3}.A, replace $\Cap ({\bf -1} ,
\{{\bf 0}, {\bf +1}\})$ by the numerator appearing in the statement of
this assertion.  The proof is identical up to formula \eqref{415}. It
remains to estimate 
\begin{equation*}
\lim_{\beta\to\infty} \sum_{\sigma\in\mf R^a} 
\bb P_{\sigma} [ H_{\bf 0}  < H_{\{{\bf -1}, {\bf +1}\}} ]
\;\;\text{which is equal to}\;\;
\lim_{\beta\to\infty} \sum_{\sigma\in\mf R^a} 
\bb P_{\sigma} [ H_{\{{\bf 0}, {\bf +1}\}}  < H_{\bf -1} ]
\end{equation*}
in view of \eqref{411}.  This expression has been computed at the end
of the proof of Proposition \ref{mt3}.A, which completes the proof of
the assertion.
\end{proof}

\begin{proof}[Proof of Proposition \ref{mt2}]
We first assume that the chain starts from $\bf 0$.  By Lemma
\ref{mest1} and Assertion \ref{as15},
\begin{equation*}
\bb E_{\bf 0} [ H_{\bf +1}]  \; = \; \big(1+o(1)\big)\, 
\frac 1{\Cap ({\bf 0}, {\bf +1})} 
\Big\{ \mu_\beta ({\bf 0}) + \mu_\beta ({\bf -1}) 
\bb P_{\bf -1} [H_{\bf 0} < H_{\bf +1}] \Big\}\;.
\end{equation*}
Since the second term in the expression inside braces is bounded by
$\mu_\beta ({\bf -1})\prec \mu_\beta ({\bf 0})$, the expectation is
equal to $(1+o(1)) \mu_\beta ({\bf 0}) / \Cap ({\bf 0}, {\bf +1})$. To
complete the proof it remains to recall \eqref{603} and Proposition
\ref{mt3}. 

We turn to the case in which the chain starts from $\bf -1$. By Lemma
\ref{mest1},
\begin{equation*}
\bb E_{\bf -1} [ H_{\bf +1}]  \; = \; \big(1+o(1)\big)\,
\frac 1{\Cap ({\bf -1}, {\bf +1})} 
\Big\{ \mu_\beta ({\bf -1}) + \mu_\beta ({\bf 0}) 
\bb P_{\bf 0} [H_{\bf -1} < H_{\bf +1}] \Big\}\;.
\end{equation*}

By \eqref{mf01}, by reversibility and by definition of the capacity,
\begin{align*}
& \mu_\beta ({\bf 0}) \, \bb P_{\bf 0} [H_{\bf -1} < H_{\bf +1}] \; =\;
\frac{M ({\bf 0}) \, \bb P_{\bf 0} [H_{\bf -1} < H^+_{\{{\bf 0}, {\bf +1}\}}] }
{\lambda_\beta ({\bf 0})\, \bb P_{\bf 0} [H_{\{{\bf -1}, {\bf +1}\}} < H^+_{\bf 0}]} \\
& \qquad =\;
\frac{M ({\bf -1}) \, \bb P_{\bf -1} [H_{\bf 0} < H^+_{\{{\bf -1}, {\bf +1}\}}] }
{\lambda_\beta ({\bf 0})\, \bb P_{\bf 0} [H_{\{{\bf -1}, {\bf +1}\}} < H^+_{\bf 0}]}
\;=\; \frac{\mu_\beta ({\bf 0}) \,  M ({\bf -1}) \, \bb P_{\bf -1} 
[H_{\bf 0} < H^+_{\{{\bf -1}, {\bf +1}\}}] }
{\Cap ( {\bf 0} ,  \{{\bf -1}, {\bf +1}\})}\;\cdot
\end{align*}
Hence, by Assertion \ref{as66},
\begin{equation*}
\bb E_{\bf -1} [ H_{\bf +1}]  \; = \; \big(1+o(1)\big)\, 
\frac {\mu_\beta ({\bf -1})}{\Cap ({\bf -1}, {\bf +1})}
\Big\{ 1 + \frac{\mu_\beta ({\bf 0}) \, \Cap ({\bf -1} , \{{\bf 0}, {\bf +1}\}) }
{\mu_\beta ({\bf -1}) \, 
\Cap ( {\bf 0} ,  \{{\bf -1}, {\bf +1}\})} \Big\}\;. 
\end{equation*}
To complete the proof it remains to recall the statements of
Proposition \ref{mt3} and Assertion \ref{as64}.
\end{proof}

\section{The hitting time of  ${\bf 0}$ starting from $\bf -1$} 
\label{sec6}

We prove in this section the identities \eqref{eq04} and
\eqref{eq03}. We start with \eqref{eq04}. By Lemma \ref{mest1} and
Assertion \ref{as15},
\begin{equation*}
\bb E_{\bf -1}[H_0] \;=\; \big(1+o(1)\big) \, 
\frac{1}{\textnormal{cap}({\bf -1},\textbf{0})} \,
\Big\{ \mu_\beta({\bf -1}) + \mu_\beta(\textbf{+1}) 
\bb P_\textbf{+1} [ H_{\bf -1} < H_\textbf{0} ] \Big\}\;.  
\end{equation*}
By \eqref{mf01} and the first identity in \eqref{mf05}, and by
reversibility, the second term inside braces is equal to
\begin{equation*}
\mu_\beta(\textbf{+1}) \, M({\bf +1})
\frac{\bb P_\textbf{+1} [ H_{\bf -1} < H^+_{\{ {\bf 0} , {\bf +1}\}} ] }
{\Cap ( {\bf +1} , \{ {\bf -1} , {\bf 0}\})}\;=\;
\mu_\beta(\textbf{+1}) \, M({\bf -1})
\frac{\bb P_{\bf -1} [ H_\textbf{+1} < H^+_{\{ {\bf -1} , {\bf 0}\}} ] }
{\Cap ( {\bf +1} , \{ {\bf -1} , {\bf 0}\})}\;\cdot
\end{equation*}
By the first identity in \eqref{mf05}, this expression is equal to
\begin{equation*}
\mu_\beta(\textbf{+1}) \,
\frac{ \textnormal{cap} ({\bf -1}, \lbrace \textbf{0}, 
\textbf{+1} \rbrace) }
{\textnormal{cap} (\textbf{+1}, \lbrace \textbf{0}, 
{\bf -1} \rbrace)} \, \bb P_{\bf -1} [H_\textbf{+1} <
H_{\textbf{0}}] \;. 
\end{equation*}
By Assertion \ref{500} and Proposition \ref{mt3}, we may replace the
ratio of the capacities by $\mu_\beta({\bf
  -1})/\mu_\beta(\textbf{0})$.  Hence,
\begin{equation*}
\bb E_{{\bf -1}} [ H_\textbf{0}] \; = \; \big(1+o(1)\big) \,
\frac{\mu_\beta({\bf -1})} {\textnormal{cap}({\bf -1},\textbf{0})} 
\Big\{  1 +  \frac{\mu_\beta(\textbf{+1})}{\mu_\beta(\textbf{0})}\,
\bb P_{{\bf -1}} [H_\textbf{+1} <H_{\textbf{0}}] \Big\} \;.
\end{equation*}
To complete the proof of \eqref{eq04}, it remains to recall the
definition of $\theta_\beta$, the statement of Assertion \ref{as18},
and the one of Lemma \ref{ll1} below.

\begin{lemma}
\label{ll1}
We have that
\begin{equation*}
\lim_{\beta\to \infty} \frac{\mu_\beta ({\bf +1})}{\mu_\beta ({\bf
    0})} \, \bb P_{\bf -1} [H_{\bf +1} < H_{\bf 0}] \;=\; \infty\;.  
\end{equation*}
\end{lemma}

The proof of this lemma is divided in several assertions.
By \eqref{mf01}, and by the definition of the capacity,
\begin{equation}
\label{cl4}
\bb P_{\bf -1} [H_{\bf +1} < H_{\bf 0}]  \;=\; 
\frac{\mu_\beta({\bf -1}) \,\lambda_\beta({\bf -1}) \, \bb P_{\bf -1}
  [H_{\bf +1} < H^+_{\{{\bf -1}, {\bf 0}\}}]}
{\Cap ({\bf -1}, \{{\bf -1}, {\bf 0}\})}\;\cdot
\end{equation}

We estimate the probability appearing in the numerator. This is done
by proposing a path from $\bf -1$ to $\bf +1$ which does not visit
$\bf 0$. The obvious path is the optimal one from $\bf -1$ to $\bf 0$
juxtaposed with the optimal one from $\bf 0$ $\bf +1$, modified not to
visit $\bf 0$.

We describe the path in $\ms S$, the set of stable configurations
introduced in \eqref{stable}. Let $\xi_0\in \ms S$ be the
configuration formed by a $L\times (L-2)$ band of $0$-spins and a
$L\times 2$ band of $-1$ spins. The first piece of the path, denoted
by $\gamma_0$, connects $\bf -1$ to $\xi_0$. It is formed by creating
and increasing a droplet of $0$-spins in a sea of $-1$-spins.

Let $ \gamma_0 = ({\bf -1} = \eta_0,\dots,\eta_{N} = \xi_0)$, where
\begin{itemize}
\item $N=2(L-3)$,
\item $\eta_1$ is a $2 \times 2$ square of $0$-spins in a background
  of negative spins,
\item For $k<N-1$, $\eta_{k+1}$ is obtained from $\eta_k$ adding a
  line of $0$-spins to transform a $j\times j$-square of $0$-spins
  into a $(j+1)\times j$-square of $0$-spins, or to transform
  $(j+1)\times j$-square of $0$-spins into a $j\times j$-square of
  $0$-spins.
\end{itemize}
Note that $\eta_N$ is obtained from $\eta_{N-1}$ transforming a
$(L-2)\times(L-2)$ rectangle into a $L\times (L-2)$ band.

Let $\xi_1\in \ms S$ be the configuration formed by a $2\times 2$
square of $+1$-spins in a background of $0$-spins. The last piece of
the path, denoted by $\gamma_1$, connects $\xi_0$ to $\bf +1$ and is
constructed in a similar way as $\gamma_0$ so that $\gamma_1 = (\xi_1
= \zeta_0, \dots, \zeta_{N} ={\bf +1})$. Note that the length of
$\gamma_1$ is the same as the one of $\gamma_0$.

Denote by $q(\eta,\xi)$ the jump probabilities of the trace of
$\sigma(t)$ on $\ms S$: $q(\eta,\xi) = \bb P_{\eta}[H_\xi = H_{\ms S
  \setminus \{\eta\}}]$. Let
\begin{equation*}
q(\gamma_0) \;=\; \prod_{k=0}^{N-1} q(\eta_k,\eta_{k+1})\; , \quad
q(\gamma_1) \;=\; \prod_{k=0}^{N-1} q(\zeta_k,\zeta_{k+1})\;, 
\end{equation*}
so that
\begin{equation}
\label{cl5}
\bb P_{\bf -1} \big [H_{\bf +1} < H^+_{\{{\bf -1}, {\bf 0}\}} \big] \;\ge\;
q(\gamma_0) \, q(\xi_0, \xi_1) \, q(\gamma_1) \;.
\end{equation}
We estimate the three terms on the right hand side.

\begin{asser}
\label{asl1}
There exists a positive constant $c_0$, independent of $\beta$, such
that 
\begin{equation*}
q(\gamma_0) \;\ge\;  c_0\,
e^{ - \beta \{ 4(n_0 -1) - [ n_0(n_0+1) - 2] h\} } \;.
\end{equation*}
\end{asser}

\begin{proof}
By the arguments presented in the proof of Assertion \ref{as10}, there
exists a positive constant $c_0$, independent of $\beta$, such that
$q(\eta_k,\eta_{k+1})\ge c_0$ if $k\ge 2n_0-3$. Thus,
\begin{equation*}
q(\gamma_0) \;\ge \; c_0 \prod_{k=0}^{2(n_0-2)} q(\eta_k,\eta_{k+1})\;,
\end{equation*}
and $\eta_{2n_0-3}$ is a $(n_0+1)\times(n_0+1)$ square of $0$ spins in
a sea of $-1$-spins. 

Denote by $\lambda_{\ms S}$ the holding rates of the trace of
$\sigma(t)$ on $\ms S$, by $\mu_{\ms S}$ the invariant probability
measure, and let $M_{\ms S} (\eta)= \lambda_{\ms S} (\eta) \mu_{\ms
  S}(\eta)$. The measure $M_{\ms S}$ is reversible for the
dicrete-time chain which jumps from $\eta$ to $\xi$ with probability
$q(\eta,\xi)$. 

By the proof of Assertion \ref{as10}, there exists a positive constant
$c_0$, independent of $\beta$, such that $q(\eta_{k+1}, \eta_k)\ge
c_0$ if $k< 2(n_0-2)$. Thus, multiplying and dividing by $M_{\ms
  S}({\bf -1})$, by reversibility
\begin{align*}
\prod_{k=0}^{2(n_0-2)} q(\eta_k,\eta_{k+1}) \; & =\; 
\frac{M_{\ms S}(\eta_{2(n_0-2)})}{M_{\ms S}({\bf -1})} \prod_{k=0}^{2n_0-5}
q(\eta_{k+1}, \eta_k) \; q(\eta_{2(n_0-2)},\eta_{2n_0-3}) \\
\; & \ge\; c_0\, \frac{M_{\ms S}(\eta_{2(n_0-2)})}{M_{\ms S}({\bf -1})}\, 
q(\eta_{2(n_0-2)},\eta_{2n_0-3})\;,
\end{align*}
where the configuration $\eta_{2(n_0-2)}$ is a $(n_0+1)\times n_0$
rectangle of $0$-spins.

Recall that $M_{\ms S} (\eta) = \mu_{\ms S} (\eta) \, \lambda_{\ms S}
(\eta)$. Since $\mu_{\ms S} (\eta) = \mu_{\beta} (\eta)/\mu_{\beta}
(\ms S)$, by \cite[Proposition 6.1]{bl2}, for any $\eta\in \ms S$,
\begin{equation}
\label{cl2}
M_{\ms S} (\eta) \;=\;
\frac{\mu_{\beta} (\eta)}{\mu_{\beta}
(\ms S)} \, \lambda _{\beta}(\eta) \, \bb P_{\eta} 
\big [H_{\ms S\setminus \{\eta\}}< H^+_{\eta} \big] \;=\; 
\frac{\Cap (\eta,\ms S\setminus \{\eta\})}{\mu_{\beta} (\ms S)}\;\cdot 
\end{equation}

We claim that
\begin{equation}
\label{cl1}
M_{\ms S} (\eta_{2(n_0-2)}) \,  q(\eta_{2(n_0-2)},\eta_{2n_0-3}) 
\;\ge\; c_0 \, \mu_{\ms S}(\eta_{2(n_0-2)}) \, e^{-\beta (2 - h)}\;. 
\end{equation}
To keep notation simple, let $\eta=\eta_{2(n_0-2)}$, $\xi =
\eta_{2n_0-3}$. By definition of $q$ and by \eqref{mf01}, the jump
probability appearing on the left hand side is equal to
\begin{equation*}
\bb P_{\eta} \big[ H_{\xi} = H_{\ms S \setminus \{\eta\}}\big] \;=\;
\frac{\mu_\beta(\eta) \, \lambda_\beta(\eta) \, \bb P_{\eta} \big[ H_{\xi} = H^+_{\ms S}\big]}
{\Cap (\eta , \ms S \setminus \{\eta\})}\;\cdot
\end{equation*}
The denominator cancels the numerator in \eqref{cl2}.
On the other hand, to reach $\xi$ from $\eta$ without returning to
$\eta$, the simplest way consists in creating a $0$-spin attached to
the longer side of the rectangle and to build a line of $0$-spins from
this first one. Only the first creation has a cost which vanishes as
$\beta\uparrow\infty$. Hence, $\lambda_\beta(\eta) \, \bb P_{\eta}
\big[ H_{\xi} = H^+_{\ms S}\big] \ge c_0 R_\beta(\eta, \eta')$ where
$\eta'$ is a critical configuration in $\mf R^l$. This completes the
proof of \eqref{cl1} since $R_\beta(\eta, \eta') =
\exp\{-\beta(2-h)\}$.

It remains to estimate $M_{\ms S}({\bf -1})$. Recall
\eqref{mf04}. Since $\xi_{{\bf -1},\ms S\setminus \{{\bf -1}\}}$ is
the configuration with three $0$-spins included in a $2\times 2$
square, $\Cap({\bf -1},\ms S\setminus \{{\bf -1}\}) \le C_0 \exp\{
-\beta [8-3h] \} \mu_\beta({\bf -1})$. Hence, by \eqref{cl2},
\begin{equation}
\label{claim2}
M_{\ms S}({\bf -1}) \;\le\; C_0\, e^{ -\beta (8-3h)} \, \mu_{\ms S}({\bf
  -1})\; .
\end{equation}

Putting together all previous estimates, we obtain that
\begin{equation*}
q(\gamma_0) \;\ge\; c_0 \, \frac{\mu_{\beta}(\eta_{2(n_0-2)})}
{\mu_\beta({\bf -1})} \,
 e^{-\beta (2 - h)} \, e^{ \beta (8-3h)} \;,
\end{equation*}
which completes the proof of the assertion in view of the definition
of $\eta_{2(n_0-2)}$.
\end{proof}

Next result is proved similarly.

\begin{asser}
\label{asl2}
There exists a positive constant $c_0$, independent of $\beta$, such
that 
\begin{equation*}
q(\gamma_1) \;\ge\;  c_0\,
e^{ - \beta \{ 4(n_0 -1) - [ n_0(n_0+1) - 2] h\} } \;.
\end{equation*}
\end{asser}

We turn to the probability $q(\xi_0, \xi_1)$. Recall that $\xi_0 $ is
the configuration formed by a $L\times (L-2)$ band of $0$-spins and a
$L\times 2$ band of $-1$ spins, and that $\xi_1$ is the configuration
formed by a $2\times 2$ square of $+1$-spins in a background of
$0$-spins.

\begin{asser}
\label{asl3}
There exists a positive constant $c_0$, independent of $\beta$, such
that 
\begin{equation*}
q(\xi_0, \xi_1) \;\ge\;  c_0\, e^{-2\beta [2-h]}\;.
\end{equation*}
\end{asser}

\begin{proof}
By definition of $q$ and by \eqref{mf01},
\begin{equation*}
q(\xi_0, \xi_1) \;=\; \bb P_{\xi_0} \big[ H_{\xi_1} = H_{\ms S
  \setminus \{\xi_0\}} \big] \;=\; \frac
{\mu_\beta(\xi_0)\, \lambda_\beta(\xi_0)\, 
\bb P_{\xi_0} \big[ H_{\xi_1} = H^+_{\ms S} \big]}
{\Cap (\xi_0 , \ms S \setminus \{\xi_0\})}\;\cdot
\end{equation*}

We claim that
\begin{equation}
\label{cl3}
\bb P_{\xi_0} \big[ H_{\xi_1} = H^+_{\ms S} \big] \;\ge\; c_0\,
e^{-2\beta [2-h]} \;.  
\end{equation}
To estimate this probability, we propose a path $\gamma_3$ from
$\xi_0$ to $\xi_1$ which avoids $\ms S$. The path consists in filling
the $-1$-spins with $0$-spins, until one $-1$-spin is left. At this
point, to avoid the configuration $\bf 0$, we switch this $-1$-spin to
$+1$. To complete the path we create a $2\times 2$ square of
$+1$-spins from the first $+1$-spin, as in the optimal path from $\bf
0$ to $\xi_1$.

Hence, $\gamma_3$ as length $2L +3$. Denote this path by
$\gamma_3=(\xi_0= \eta'_0, \eta'_1, \dots, \eta'_{2L +
  3}=\xi_1)$. From $\eta'_0$ to $\eta'_{2L-2}$ the next configurations
is obtained by flipping a $-1$ spin to a $0$-spin as in an optimal
path from $\xi_0$ to $\bf 0$. In this piece of the path, all jumps
have a probability bounded below by a positive constant. Therefore,
there exists a positive constant $c_0$, independent of $\beta$, such
that
\begin{equation*}
\bb P_{\xi_0} \big[ H_{\xi_1} = H^+_{\ms S} \big] \;\ge\;
\prod_{j=0}^{2L+2} p_\beta(\eta'_j, \eta'_{j+1}) \;\ge\;
c_0  \, \prod_{j=2L-1}^{2L+2} p_\beta(\eta'_j, \eta'_{j+1}) \;. 
\end{equation*}
The first and the last probabilities in this product,
$p_\beta(\eta'_{2L-1}, \eta'_{2L})$ and $p_\beta(\eta'_{2L+2},
\eta'_{2L+3})$, are also bounded below by a positive constant. The
other ones can be estimated easily, proving \eqref{cl3}.

By \eqref{mf04}, $\Cap (\xi_0 , \ms S \setminus
\{\xi_0\})/\mu_\beta(\xi_0)$ is bounded above by $C_0 \exp \{-\beta
[2-h]\}$, while an elementary computation shows that
$\lambda_\beta(\xi_0)$ is bounded below by $c_0 \exp \{-\beta
[2-h]\}$. This completes the proof of the assertion.
\end{proof}

By \eqref{cl5} and Assertions \ref{asl1}, \ref{asl2} and \ref{asl3}, 
\begin{equation}
\label{cl6}
\bb P_{\bf -1} \big [H_{\bf +1} < H^+_{\{{\bf -1}, {\bf 0}\}} \big]
\;\ge\; c_0\, e^{ - 2 \beta \{ 2(2n_0 -1) - [ n_0(n_0+1) - 1] h\} } \;.
\end{equation}

\begin{proof}[Proof of Lemma \ref{ll1}]
Since $\lambda({\bf -1}) \ge c_0 e^{-\beta [4-h]}$, by \eqref{cl4},
\eqref{cl6},
\begin{equation*}
\bb P_{\bf -1} [H_{\bf +1} < H_{\bf 0}]  \;\ge \; c_0\, 
\frac{\mu_\beta({\bf -1})}
{\Cap ({\bf -1}, \{{\bf -1}, {\bf 0}\})}\,
e^{ - \beta \{ 8 n_0 - [ 2 n_0(n_0+1) - 1] h\} }\;.
\end{equation*}
Therefore, by Proposition \ref{mt3},
\begin{equation*}
\frac{\mu_\beta ({\bf +1})}{\mu_\beta ({\bf
    0})} \, \bb P_{\bf -1} [H_{\bf +1} < H_{\bf 0}] \;\ge\;
c_0\, e^{\beta L^2}\, 
e^{ - \beta \{ 4 (n_0-1)  - [ n_0(n_0+1) - 2] h\} }\;.
\end{equation*}

It remains to show that $L^2 > 4 (n_0-1)  - [ n_0(n_0+1) - 2] h$.
By definition of $n_0$, $n_0 h >2-h$, so that $4 (n_0-1)  - [
n_0(n_0+1) - 2] h \le 2n_0 - 6 + hn_0 +3h$. As $hn_0<2$ and $h< 1$,
this expression is less than or equal to $2n_0$. This expression is
smaller than $L^2$ because $L\ge 2$ and $L>n_0$.
\end{proof}	

\section{Proof of Theorem \ref{mt1b}}

The statement of Theorem \ref{mt1b} follows from Propositions
\ref{c-trb} and \ref{neglb} below and from Theorem 5.1 in
\cite{l-soft}.  We start deriving some consequences of the assumption
\eqref{cond}. Clearly, it follows from \eqref{cond} and from
\eqref{mf04} that for all $\eta\in\ms M$ and $\sigma\in\ms V_\eta$,
$\sigma\not =\eta$,
\begin{equation}
\label{cond3}
\frac{\mu_\beta(\eta)}{\Cap (\sigma, \eta)} \;\prec\; \theta_\beta\;.
\end{equation}

\begin{asser}
\label{as52}
For all $\eta\in\ms M$ and $\sigma\in\ms V_\eta$, $\sigma\not =\eta$,
\begin{equation*}
\mu_\beta(\sigma) \;\prec\; \mu_\beta(\eta) \;, \quad
\Cap (\eta , \ms M \setminus \{\eta\}) \; \prec\; \Cap (\sigma, \eta)
\;.
\end{equation*}
\end{asser}

\begin{proof}
The first bound is a straightforward consequence of the hypothesis
$\bb H(\sigma)>\bb H(\eta)$. In view of \eqref{mf04}, to prove the
second bound we have to show that $\bb H (\sigma, \eta) < \bb H(\eta ,
\ms M \setminus \{\eta\})$. The case $\eta= {\bf -1}$ is a consequence
of the second hypothesis in \eqref{cond}, as the case $\eta= {\bf 0}$
if one recalls \eqref{200}. It remains to consider the case $\eta=
{\bf +1}$. By condition \eqref{cond} and \eqref{200},
\begin{equation*}
\bb H (\sigma, {\bf +1}) \;<\; \bb H({\bf +1}) \;+\; 
\bb H({\bf 0} , \{{\bf -1},{\bf  +1}\}) - \bb H({\bf 0})
\;<\; \bb H({\bf 0} , \{{\bf -1},{\bf  +1}\}) \;.
\end{equation*}
By Assertion \ref{500} and by \eqref{mf04}, $\bb H({\bf 0} , \{{\bf
  -1},{\bf +1}\}) = \bb H({\bf +1}, \{{\bf -1},{\bf 0}\})$. Therefore,
\begin{equation*}
\bb H (\sigma, {\bf +1}) \;<\; \bb H({\bf +1}, \{{\bf -1},{\bf 0}\})\;,
\end{equation*}
which proves the second claim of the assertion.
\end{proof}

\begin{asser}
\label{as51}
For all $\eta\not = \xi\in \ms M$ and for all $\sigma\in \ms
V_{\eta}$, $\sigma'\in \ms V_{\xi}$,
\begin{equation*}
\Cap (\sigma , \sigma') \;\approx\; \Cap (\eta, \xi) \;.
\end{equation*}
\end{asser}

\begin{proof}
Fix $\eta\not = \xi\in \ms M$ and $\sigma\in \ms V_{\eta}$,
$\sigma'\in \ms V_{\xi}$.  We need to prove that $\bb H (\eta, \xi) =
\bb H(\sigma , \sigma')$.  On the one hand, by definition, $\bb
H(\sigma , \sigma') \le \max \{\bb H(\sigma , \eta), \bb H(\eta ,
\eta'), \bb H(\eta' , \sigma') \}$. By the proof of Assertion
\ref{as52}, $\bb H(\sigma , \eta)< \bb H(\eta ,\ms M \setminus
\{\eta\})$, with a similar inequality replacing $\sigma$, $\eta$ by
$\sigma'$, $\eta'$, respectively. Since $\bb H(\ms A, \ms B)$ is
decreasing in each variable, $\bb H(\eta ,\ms M \setminus \{\eta\})$
and $\bb H(\eta' ,\ms M \setminus \{\eta'\})$ are less than or equal
to $\bb H(\eta ,\eta')$, which shows that $\bb H(\sigma , \sigma') \le
\bb H(\eta , \eta')$.

Conversely, $\bb H(\eta , \eta') \le \max \{\bb H(\eta, \sigma), \bb
H(\sigma , \sigma'), \bb H(\sigma' , \eta') \}$. By the previous
paragraph, $\bb H(\eta, \sigma) < \bb H(\eta ,\eta')$ and $\bb
H(\sigma' , \eta') < \bb H(\eta ,\eta')$ so that $\bb H(\eta , \eta')
\le \bb H(\sigma , \sigma')$. This completes the proof of the
assertion. 
\end{proof}

We conclude this preamble with two simple remarks.  Fix $\eta\in\ms M$
and $\sigma\in\ms V_\eta$. By \eqref{mf03} and Assertion \ref{as51},
\begin{equation*}
\Cap (\sigma, \cup_{\xi\not = \eta} \ms V_\xi) \;\approx\; 
\max_{\sigma'\in \cup_{\xi\not = \eta} \ms V_\xi} \Cap (\sigma,
\sigma') \;\approx\; \max_{\xi\in \ms M \setminus\{ \eta\}}
\Cap (\eta, \xi)\;.
\end{equation*}
Applying  \eqref{mf03} once more, we conclude that
\begin{equation}
\label{cons2}
\Cap (\sigma, \cup_{\xi\not = \eta} \ms V_\xi) \;\approx\;
\Cap (\eta, \ms M \setminus\{ \eta\})\;.  
\end{equation}
In particular, by Assertion \ref{as52},
\begin{equation}
\label{cons1}
\lim_{\beta\to\infty} \frac{\Cap (\sigma, \cup_{\xi\not = \eta} \ms
  V_\xi) }{\Cap (\sigma, \eta)}\;=\; 0\;.
\end{equation}

Denote by $\sigma^{\ms A}(t)$, $\ms A\subset\Omega$, the trace of
$\sigma(t)$ on $\ms A$.  By \cite[Proposition 6.1]{bl2}, $\sigma^{\ms
  A}(t)$ is a continuous-time Markov chain. Moreover, for $\ms
B\subset\ms A$, $\sigma^{\ms B}(t)$ is the trace of $\sigma^{\ms
  A}(t)$ on $\ms B$. When $\ms A = \ms M$, we represent $\sigma^{\ms
  A}(t)$ by $\eta(t)$.  Denote $R^{\ms A}_\beta(\sigma, \sigma')$,
$\sigma\not = \sigma'\in \ms A$, the jump rates of the Markov chain
$\sigma^{\ms A}(t)$.

Recall the definition of the map $\pi:\ms M \to \{-1,0,1\}$,
introduced just before the statement of Theorem \ref{mt1b}.
Denote by $\psi = \psi_{\ms V} : \ms V \to \{-1,0, 1\}$
the projections defined by $\psi(\sigma)=\pi(\eta)$ if
$\sigma\in \ms V_\eta$:
\begin{equation*}
\psi(\sigma) \;=\; \sum_{\eta \in \ms M} \pi(\eta) \, \mb 1\{\sigma \in \ms
V_\eta\} \;.
\end{equation*}
Recall also the definition of the time-scale $\theta_\beta$ introduced
in \eqref{203}.

\begin{proposition}
\label{c-trb}
As $\beta\uparrow\infty$, the speeded-up, hidden Markov chain
$\psi(\sigma^{\ms V}(\theta_\beta t))$ converges to the
continuous-time Markov chain $X(t)$ introduced in Theorem \ref{mt1b}.
\end{proposition}

We first prove Proposition \ref{c-trb} in the case where the wells
$\ms V_\eta$ are singletons: $\ms V_\eta = \{\eta\}$. In this case,
$\psi$ is a bijection, and $\psi(\eta(t))$ is a Markov chain on
$\{-1,0,1\}$.

\begin{lemma}
\label{c-tr}
As $\beta\uparrow\infty$, the speeded-up chain $\eta(\theta_\beta t)$
converges to the continuous-time Markov chain on $\ms M$ in which $\bf
+1$ is an absorbing state, and whose jump rates $\mb r(\eta, \xi)$, are
given by
\begin{equation*}
\mb r({\bf -1}, {\bf 0}) \;=\; \mb  r({\bf 0},{\bf +1})  \;=\; 1 \;,
\quad \mb r({\bf -1},{\bf +1}) \;=\; \mb r({\bf 0},{\bf -1})\;=\;
0\;. 
\end{equation*}
\end{lemma}

\begin{proof}
Denote by $r_\beta(\eta,\xi)$ the jump rates of the chain
$\eta(\theta_\beta t)$. It is enough to prove that
\begin{equation}
\label{lim}
\lim_{\beta\to\infty} r_\beta(\eta,\xi) \;=\; \mb r(\eta,\xi)
\end{equation}
for all $\eta\not =\xi\in\ms M$.

By \cite[Proposition 6.1]{bl2}, the jump rates $r_\beta (\eta,\xi)$,
$\eta\not =\xi\in \ms M$, of the Markov chain $\eta_\beta(t)$ are
given by
\begin{equation*}
r_\beta(\eta,\xi) \;=\; \theta_\beta \, \lambda(\eta) \, 
\bb P_{\eta} [H_\xi = H^+_{\ms M}]\;. 
\end{equation*}
Dividing and multiplying the previous expression by $\bb P_{\eta}
[H_{\ms M \setminus\{\eta\}} < H^+_{\eta}]$, in view of \cite[Lemma
6.6]{bl2} and of \eqref{mf01}, we obtain that
\begin{equation*}
r_\beta(\eta,\xi) \;=\; \frac{\theta_\beta}{\mu_\beta(\eta)} 
\, \Cap(\eta, \ms M\setminus\{\eta\})\; 
\bb P_{\eta} [H_\xi < H_{\ms M \setminus \{\eta,\xi\}}]\;. 
\end{equation*}

For $\eta= {\bf +1}$ and $\xi = {\bf -1}$, ${\bf 0}$, by Assertion
\ref{500} and by Proposition \ref{mt3},
\begin{equation*}
\lim_{\beta\to\infty} r_\beta({\bf +1},\xi) \;\le\;
\lim_{\beta\to\infty} \frac{\theta_\beta}{\mu_\beta({\bf +1})} 
\, \Cap({\bf +1}, \ms M\setminus\{{\bf +1}\}) 
\; = \; \lim_{\beta\to\infty} \frac
{\mu_\beta({\bf 0})}{\mu_\beta({\bf +1})} 
\;=\; 0\;.
\end{equation*}
On the other hand, by Proposition \ref{mt3}, 
\begin{equation*}
\lim_{\beta\to\infty} \frac{\theta_\beta}{\mu_\beta({\bf 0})} 
\, \Cap({\bf 0}, \ms M\setminus\{{\bf 0}\}) \;=\; 1\;,
\end{equation*}
while $\theta_\beta \Cap({\bf -1}, \ms M\setminus\{{\bf -1}\})/
\mu_\beta({\bf 0}) =1$.  Furthermore, by Proposition \ref{0b1} and
Assertion \ref{as65},
\begin{equation*}
\lim_{\beta\to\infty} \bb P_{\bf -1} [H_{\bf +1} < H_{\bf 0}]\;=\;
\lim_{\beta\to\infty} \bb P_{\bf 0} [H_{\bf -1} < H_{\bf +1}]\;=\;
0\;. 
\end{equation*}
This yields \eqref{lim} and completes the proof of the lemma.
\end{proof}

Denote by $\bb P^{\ms V}_\sigma$, $\sigma\in \ms V$, the probability
measure on the path space $D(\bb R_+, \ms V)$ induced by the Markov
chain $\sigma^{\ms V}(t)$ starting from $\sigma$. Expectation with
respect to $\bb P^{\ms V}_\sigma$ is represented by $\bb E^{\ms
  V}_\sigma$. Clearly, for any disjoint subsets $\ms A$, $\ms B$ of
$\ms V$,
\begin{equation}
\label{eq1}
\bb P^{\ms V}_\sigma [H_{\ms A} < H_{\ms B}]  \;=\; \bb P_\sigma
[H_{\ms A} < H_{\ms B}]\;.  
\end{equation}

The hitting time of a subset $\ms A$ of $\ms V$ by the trace chain
$\sigma^{\ms V}$ can be represented in terms of the original chain 
$\sigma(t)$. Under $\bb P_\sigma$,
\begin{equation}
\label{eq2}
H^{\ms V}_{\ms A} \;=\; \inf\{t>0 :  \sigma^{\ms V}(t)\in\ms A\}
\;=\; \int_0^{H_{\ms A}} \mb 1\{\sigma(t)\in \ms V\}\, dt \;.
\end{equation}

Let
\begin{equation*}
\breve {\ms V}(\eta) \;=\; \breve {\ms V}_\eta \;=\; 
\bigcup_{\zeta\not = \eta} \ms V_\zeta\;.
\end{equation*}
Denote by $\{T_j : j \ge 0\}$ the jump times of the hidden chain
$\psi(\sigma^{\ms V}(t))$:
\begin{equation*}
T_0 \;=\; 0\;, \quad T_{j+1} \;=\; \inf\{ t \ge T_j : 
\sigma^{\ms V}(t) \in \breve {\ms V}(\sigma^{\ms V}(T_j))\big\}\;, 
\quad j\ge 0\;,
\end{equation*}
Similarly, denote by $\{\tau_j : j \ge 0\}$ the successive jump times
of the chain $\eta(t)$.

\begin{lemma}
\label{coupling}
Fix $\sigma\in \ms V_{\bf -1}$.  There exists a sequence
$\epsilon_\beta\to 0$ such that for $j=1$, $2$
\begin{equation}
\label{as1}
\lim_{\beta\to\infty} \bb P^{\ms V}_{\sigma} \big[ \, | \tau_j - T_j | \ge
\theta_\beta \epsilon_\beta \big] \;=\;0\;, \quad
\lim_{\beta\to\infty} \bb P^{\ms V}_{\sigma} \big[ \, T_3-T_2  \le
\theta_\beta \epsilon^{-1}_\beta \big] \;=\;0\;.
\end{equation}
Moreover,
\begin{equation}
\label{as2}
\lim_{\beta\to\infty} \bb P^{\ms V}_{\sigma} \big[ \sigma(T_1) \not \in
\ms V_{\bf 0} \big] \;=\;0\;, \quad
\lim_{\beta\to\infty} \bb P^{\ms V}_{\sigma} \big[ \sigma(T_2) \not \in
\ms V_{\bf +1} \big] \;=\;0\;.
\end{equation}
\end{lemma}

\begin{proof}
Fix a configuration $\sigma\in \ms V_{\bf -1}$. By \eqref{eq1}, \eqref{mf02} and
\eqref{cons1}, 
\begin{equation}
\label{claim1}
\limsup_{\beta\to\infty} \bb P^{\ms V}_{\sigma} \big[ H_{\breve {\ms  V}(\bf -1)}
< H_{\bf -1} \big] \;\le\; \lim_{\beta\to\infty}
\frac{\Cap (\sigma, \breve {\ms  V}(\bf -1))}{\Cap (\sigma ,  {\bf -1})}
\;=\;0\;.
\end{equation}

On the other hand, under $\bb P^{\ms V}_{\sigma}$,
\begin{equation*}
\tau_1 \;=\; \int_{H_{\bf -1}}^{H_{{\bf 0}, {\bf +1}}} \mb 1 \{\sigma(s) =
{\bf -1}\} \, ds\;.
\end{equation*}
Hence, under $\bb P^{\ms V}_{\sigma}$ and on the event $\{H_{\bf -1} <
H_{\breve {\ms V}(\bf -1)} \}$ we have that
\begin{align*}
T_1 \; & =\; H_{\bf -1} \;+\; \int_{H_{\bf -1}}^{H(\breve {\ms V}_{\bf
    -1})} \big\{ \mb 1 \{\sigma(s) =
{\bf -1}\} + \mb 1 \{\sigma(s) \not = {\bf -1}\} \big\}\, ds \\
\; & =\; \tau_1 \;+\; H_{\bf -1} \;+\; \int_{H_{\bf -1}}^{H(\breve {\ms V}_{\bf -1})} 
\mb 1 \{\sigma(s) \not = {\bf -1}\} \, ds
\;-\; \int_{H(\breve {\ms V}_{\bf -1})}^{H_{{\bf 0}, {\bf +1}}}
\mb 1 \{\sigma(s) = {\bf -1}\} \, ds\;.
\end{align*}
It remains to estimate the last three terms.

To bound the first term, by \eqref{eq1}, \eqref{mf02}, and \eqref{cons1},
\begin{equation*}
\limsup_{\beta\to\infty} \bb P^{\ms V}_\sigma [ H_{\breve {\ms V}_{\bf -1}} < H_{\bf -1} ] 
\;=\; \limsup_{\beta\to\infty} \bb P_\sigma [ H_{\breve {\ms V}_{\bf -1}} < H_{\bf -1} ] 
\;\le\; \lim_{\beta\to\infty} \frac{\Cap (\sigma ,
  \breve {\ms V}_{\bf -1})}{\Cap (\sigma, {\bf -1})} \;=\; 0\;.
\end{equation*}
Hence, to prove that $\bb P^{\ms V}_\sigma [ H_{\bf -1} > \theta_\beta
\epsilon_\beta]\to 0$, it is enough to show that
\begin{equation*}
\lim_{\beta\to\infty} \bb P^{\ms V}_\sigma [ H_{\breve {\ms V}_{\bf
    -1} \cup \{{\bf -1}\}} > \theta_\beta \epsilon_\beta]\;=\; 0\;.
\end{equation*}
By Tchebycheff inequality, by Lemma 6.9 and Proposition 6.10 in
\cite{bl2}, and by \eqref{eq1}, the previous probability is less than
or equal to
\begin{equation*}
\frac 1{\theta_\beta \epsilon_\beta} \,
\frac 1{\Cap(\sigma, \breve {\ms V}_{\bf -1} \cup \{{\bf -1}\})} \,
\sum_{\eta\in\ms V} \mu_\beta(\eta) \, \bb P_\eta [ H_\sigma < 
H_{\breve {\ms V}_{\bf -1}  \cup \{{\bf -1}\}} ] \;.
\end{equation*}
By definition of $\theta_\beta$, since the capacity is monotone, and
since we may restrict the sum to $\ms V_{\bf -1}$, the previous
expression is less than or equal to
\begin{equation*}
\frac 1{\epsilon_\beta} \, 
\frac{\Cap({\bf -1},  \{{\bf 0}, {\bf +1}\})}
{\Cap(\sigma, {\bf -1})} \, 
\frac{\mu_\beta(\ms V_{-1})}{\mu_\beta({\bf -1})} \;\cdot
\end{equation*}
By Assertion \ref{as52}, $\mu_\beta(\ms V_{-1})/\mu_\beta({\bf -1})$ is
bounded and $\Cap({\bf -1}, \{{\bf 0}, {\bf +1}\})/ \Cap(\sigma, {\bf
  -1})$ vanishes as $\beta\uparrow\infty$. Hence, the previous
expression converges to $0$ for every sequence $\epsilon_\beta$ which
does not decrease too fast.

We turn to the second term of the decomposition of $T_1$.  By the
strong Markov property and by \eqref{eq2}, we need to estimate,
\begin{align*}
& \bb P^{\ms V}_{\bf -1} \Big[  \int_{0}^{H(\breve {\ms V}_{\bf -1})} 
\mb 1 \{\sigma(s) \not = {\bf -1}\} \, ds > 
\theta_\beta \epsilon_\beta\Big] \\
&\quad \;=\;
\bb P_{\bf -1} \Big[  \int_{0}^{H(\breve {\ms V}_{\bf -1})} 
\mb 1 \{\sigma(s) \in \ms V \setminus \{{\bf -1}\} \} \, ds > 
\theta_\beta \epsilon_\beta\Big] \;.
\end{align*}
By Tchebycheff inequality and by \cite[Proposition 6.10]{bl2}, the
previous probability is less than or equal to
\begin{equation*}
\frac 1{\theta_\beta \epsilon_\beta} \,
\frac 1{\Cap({\bf -1}, \breve {\ms V}_{\bf -1})} \,
\sum_{\eta\in \ms V \setminus \{{\bf -1}\}} \mu_\beta(\eta) \, \bb P_\eta [ H_{\bf -1} < 
H_{\breve {\ms V}_{\bf -1}} ] \;.
\end{equation*}
By \eqref{cons2}, $\Cap({\bf -1}, \breve {\ms V}_{\bf -1}) \approx
\Cap({\bf -1}, \{{\bf 0}, {\bf +1}\})$. Hence, by definition of
$\theta_\beta$, and since the sum can be restricted to the set ${\ms
  V}_{\bf -1} \setminus \{{\bf -1}\}$, the previous expression is less
than or equal to
\begin{equation*}
\frac {C_0}{\epsilon_\beta} \, \frac 1{\mu_\beta({\bf -1})} \,
\mu_\beta({\ms V}_{\bf -1} \setminus \{{\bf -1}\}) 
\end{equation*}
for some finite constant $C_0$. By Assertion \ref{as52}, the ratio of
the measures vanishes as $\beta\uparrow\infty$. In particular, the
previous expression converges to $0$ as $\beta\uparrow\infty$ if
$\epsilon_\beta$ does not decrease too fast.

The third term in the decomposition of $T_1$ is absolutely bounded by
$H_{{\bf 0}, {\bf +1}} - H(\breve {\ms V}_\eta)$ and can be handled as
the first one. This proves the first assertion of \eqref{as1} for
$j=1$. 

In a similar way one proves that $T_2-T_1$ is close to
$\tau_2-\tau_1$. The first assertion of \eqref{as1} for $j=2$ follows
from this result and from the bound for $T_1-\tau_1$.  The details are
left to the reader.

We turn to the proof of the first assertion in \eqref{as2}. Since
$\breve {\ms V}_{{\bf -1}} = {\ms V}_{{\bf 0}} \cup {\ms V}_{{\bf
    +1}}$, 
\begin{equation*}
\bb P^{\ms V}_{\sigma} \big[ \sigma(T_1) \not \in
\ms V_{\bf 0} \big] \;=\; \bb P_{\sigma} \big[ \sigma(H) \in
\ms V_{\bf +1} \big]\;,
\end{equation*}
where $H = H(\breve {\ms V}_{{\bf -1}})$. We may rewrite the previous
probability as
\begin{equation*}
\bb P_{\sigma} \big[ \sigma(H) \in \ms V_{\bf +1} \,,\, H_{\bf 0} < H_{\bf +1} \big]
\;+\; 
\bb P_{\sigma} \big[ \sigma(H) \in \ms V_{\bf +1} \,,\, H_{\bf 0} >
H_{\bf +1} \big] \;.
\end{equation*}
Both expression vanishes as $\beta\uparrow\infty$. The second one is
bounded by $\bb P_{\sigma} [ H_{\bf +1} < H_{\bf 0}]$, which vanishes
by Proposition \ref{0b1}. Since $H < \min\{ H_{\bf 0} , H_{\bf +1}\}$, by
the strong Markov property, the first term is less than or equal to
\begin{equation*}
\max_{\sigma'\in \ms V_{\bf +1}}
\bb P_{\sigma'} \big[ H_{\bf 0} < H_{\bf +1} \big]\;.
\end{equation*}
This expression converges to $0$ as $\beta\uparrow\infty$ because $\ms
V_{\bf +1}$ is contained in the basin of attraction of ${\bf +1}$. The
proof of the second assertion in \eqref{as2} is similar and left to
the reader.

We finally examine the third assertion of \eqref{as1}. By the second
assertion of \eqref{as2}, it is enough to prove that 
\begin{equation*}
\lim_{\beta\to\infty} \bb P^{\ms V}_{\sigma} \big[ \, T_3-T_2  \le
\theta_\beta \epsilon^{-1}_\beta \,,\, \sigma(T_2) \in
\ms V_{\bf +1} \big] \;=\;0\;.
\end{equation*}
By the strong Markov property, this limit holds if
\begin{equation*}
\lim_{\beta\to\infty} \max_{\sigma'\in \ms V_{\bf +1}} 
\bb P^{\ms V}_{\sigma'} \big[ \, T_1  \le
\theta_\beta \epsilon^{-1}_\beta \big] \;=\;0\;.
\end{equation*}
Since $\ms V_{\bf +1}$ is contained in the basin of attraction of $\bf
+1$, it is enough to show that
\begin{equation*}
\lim_{\beta\to\infty} \max_{\sigma'\in \ms V_{\bf +1}} 
\bb P^{\ms V}_{\sigma'} \big[ \, T_1  \le
\theta_\beta \epsilon^{-1}_\beta \,,\, H_{\bf +1} 
< T_1 \big] \;=\;0\;.
\end{equation*}
On the event $\{H_{\bf +1} < T_1\}$, $\{T_1 \le \theta_\beta
\epsilon^{-1}_\beta \}\subset \{T_1 \circ \theta_{H_{\bf +1}} \le
\theta_\beta \epsilon^{-1}_\beta\}$, where $\{\theta_t : t\ge 0\}$
represents the semigroup of time translations. In particular, by the
strong Markov property, we just need to show that
\begin{equation*}
\lim_{\beta\to\infty} 
\bb P^{\ms V}_{\bf +1} \big[ \, T_1  \le
\theta_\beta \epsilon^{-1}_\beta \big] \;=\;0\;.
\end{equation*}

Let $\{\mf e_j : j\ge 1\}$ be the length of the sojourn times at $\bf
+1$. Hence, $\{\mf e_j : j\ge 1\}$ is a sequence of i.i.d. exponential
random variables with parameter $\lambda({\bf +1})$. Denote by $\ms A$
the set of configurations with at least $n_0(n_0+1)$ sites with spins
not equal to $+1$. Each time the process leaves the state $\bf +1$ it
attempts to reach $\ms A$ before it returns to $\bf +1$. Let
$\delta$ be the probability of success:
\begin{equation*}
\delta \;=\; \bb P^{\ms V}_{\bf +1} \big[ \, H_{\ms A}  <
H^+_{\bf +1} \big]\;.
\end{equation*}
Let $N\ge 1$ be the number of attempts up to the first success so that
$\sum_{1\le j\le N} \mf e_j$ represents the total time the process
$\sigma(t)$ remained at $\bf +1$ before it reached $\ms A$. It is
clear that under $\bb P^{\ms V}_{\bf +1}$,
\begin{equation*}
\sum_{j=1}^N \mf e_j \;\le\; T_1  \;,
\end{equation*}
and that $N$ is a geometric random variable of parameter $\delta$
independent of the sequence $\{\mf e_j : j\ge 1\}$. In view of the
previous inequality, it is enough to prove that
\begin{equation*}
\lim_{\beta\to\infty} 
\bb P_{\bf +1} \big[ \sum_{j=1}^N \mf e_j  \le
\theta_\beta \epsilon^{-1}_\beta \big] \;=\;0\;.  
\end{equation*}

The previous probability is less than or equal to
\begin{equation*}
\lambda({\bf +1}) \, \theta_\beta \epsilon^{-1}_\beta  
\, \bb P^{\ms V}_{\bf +1} \big[ \, H_{\ms A}  < H^+_{\bf +1} \big]
\;=\; \frac 1{\epsilon_\beta} \, \frac{\theta_\beta}{\mu_\beta({\bf +1}) }
\, \Cap (\ms A, {\bf +1})\;.
\end{equation*}
Since $\theta_\beta \Cap (\ms A, {\bf +1})/\mu_\beta({\bf +1}) \prec
1$, the previous expression vanishes if $\epsilon_\beta$ does not
decrease too fast to $0$. This completes the proof of the lemma.
\end{proof}

\begin{proof}[Proof of Proposition \ref{c-trb}]
The assertion of the proposition is a straightforward consequence of
Lemmas \ref{c-tr} and \ref{coupling}. 

Fix $\sigma\in\ms V_{\bf -1}$ and recall the notation introduced in
Lemma \ref{coupling}.  Let $\ms A = \{\sigma(T_1) \in \ms V_{\bf 0} \}
\cap \{\sigma(T_2) \in \ms V_{\bf +1} \}$. By \eqref{as2}, $\bb P^{\ms
  V}_\sigma [\ms A^c]\to 0$. On the set $\ms A$,
\begin{equation*}
\psi(\sigma^{\ms V}(\theta_\beta t)) \;=\;
- \mb 1\{t < T_1/\theta_\beta\} \;+\; 
\mb 1\{T_2/\theta_\beta \le t < T_3/\theta_\beta\}
\end{equation*}
By Lemma \ref{c-tr}, $(\tau_1/\theta_\beta,
(\tau_2-\tau_1)/\theta_\beta)$ converges to a pair of independent,
mean $1$, exponential random variables. Hence, by \eqref{as1},
$(T_1/\theta_\beta, (T_2-T_1)/\theta_\beta, (T_3-T_2)/\theta_\beta)$
converges in distribution to $(\mf e_1, \mf e_2, \infty)$, where $(\mf
e_1, \mf e_2)$ is a pair of independent, mean $1$, exponential random
variables. This completes the proof.
\end{proof}

\begin{lemma}
\label{neglb}
Let $\Delta = \Omega \setminus \ms V$. For all $\xi\in \ms V$, $t>0$, 
\begin{equation*}
\lim_{\beta\to\infty} \bb E_\xi \Big[ \int_0^t \mb
1\{\sigma(s\theta_\beta) \in \Delta\}\, ds \Big]\;=\; 0\;. 
\end{equation*}
\end{lemma}

\begin{proof}
Fix $\xi\in \ms V_{\bf +1}$.
On the one hand, by \cite[Proposition 6.10]{bl2},
\begin{equation*}
\frac 1{\theta_\beta} \bb E_\xi \Big[ \int_0^{H_{\bf +1}} \mb
1\{\sigma(s) \in \Delta\}\, ds \Big] \;\le\;
\frac 1{\theta_\beta} \,\frac {\mu_\beta({\bf +1})} {\Cap (\xi, {\bf +1})}\,
\,\frac {\mu_\beta(\Delta)}{\mu_\beta({\bf +1})}\;\cdot
\end{equation*}
This expression vanishes as $\beta\uparrow\infty$ because, by
\eqref{cond3}, $\mu_\beta({\bf +1}) /\Cap (\xi, {\bf +1}) \preceq
\theta_\beta$, and because $\mu_\beta(\Delta) \prec \mu_\beta({\bf
  +1})$, as ${\bf +1}$ is the unique ground state.

On the other hand, by the strong Markov property,
\begin{equation*}
\frac 1{\theta_\beta} \bb E_\xi \Big[ \int_{H_{\bf
    +1}}^{t\theta_\beta} 
\mb 1\{\sigma(s) \in \Delta\}\, ds \Big] \;\le\;
\frac 1{\theta_\beta} \bb E_{\bf +1} \Big[ \int_{0}^{t\theta_\beta} 
\mb 1\{\sigma(s) \in \Delta\}\, ds \Big]\;.
\end{equation*}
Therefore, to prove the lemma for $\xi\in \ms V_{\bf +1}$ it is enough
to show that
\begin{equation*}
\lim_{\beta\to\infty} \bb E_{\bf +1} \Big[ \int_0^t \mb
1\{\sigma(s\theta_\beta) \in \Delta\}\, ds \Big]\;=\; 0\;. 
\end{equation*}
This last assertion follows from Lemma \ref{negl} below.

Similar arguments permit to reduce the statement of the lemma for
$\xi\in\ms V_{\bf 0}$ (resp. $\xi\in\ms V_{\bf -1}$) to the
verification that
\begin{equation*}
\lim_{\beta\to\infty} \bb E_{\zeta} \Big[ \int_0^t \mb
1\{\sigma(s\theta_\beta) \in \Delta\}\, ds \Big]\;=\; 0\;, 
\end{equation*}
for $\zeta= {\bf 0}$ (resp. $\zeta= {\bf -1}$), which follows from
Lemma \ref{negl} below.

\end{proof}

\begin{lemma}
\label{negl}
Let $\Delta^* = \Omega \setminus \ms M$. For all $\xi\in \ms M$, $t>0$, 
\begin{equation*}
\lim_{\beta\to\infty} \bb E_\xi \Big[ \int_0^t \mb
1\{\sigma(s\theta_\beta) \in \Delta^*\}\, ds \Big]\;=\; 0\;. 
\end{equation*}
\end{lemma}

\begin{proof}
Consider first the case $\xi = {\bf +1}$. Clearly,
\begin{equation*}
\bb E_{\bf +1} \Big[ \int_0^t \mb
1\{\sigma(s\theta_\beta) \in \Delta^*\}\, ds \Big]
\;\le\; \frac 1{\mu_\beta({\bf +1})}
\sum_{\sigma\in\Omega} \mu_\beta(\sigma)\,
\bb E_{\sigma} \Big[ \int_0^t \mb
1\{\sigma(s\theta_\beta) \in \Delta^*\}\, ds \Big]\;.
\end{equation*}
Since $\mu_\beta$ is the stationary state, the previous expression is
equal to
\begin{equation*}
\frac {t \mu_\beta(\Delta^*)}{\mu_\beta({\bf +1})} \;,
\end{equation*}
which vanishes as $\beta\to\infty$ because $\bf +1$ is the unique
ground state.

We turn to the case $\xi = {\bf 0}$. We first claim that
\begin{equation}
\label{ea1}
\lim_{\beta\to\infty} \bb E_{\bf 0} \Big[ 
\frac 1{\theta_\beta} \int_0^{H_{\{{\bf -1}, {\bf +1}\}}} \mb
1\{\sigma(s) \in \Delta^*\}\, ds \Big] \;=\; 0
\;.
\end{equation}
Indeed, by \cite[Proposition 6.10]{bl2}, the previous expectation is
equal to 
\begin{equation*}
\frac 1{\theta_\beta}\,
\frac{\< V \mb 1\{\Delta^*\}\>_{\mu_\beta}}
{\Cap ({\bf 0} , \{{\bf -1}, {\bf +1}\} )} \;=\;
\frac {\mu_\beta({\bf 0})}{\theta_\beta\, 
\Cap ({\bf 0} , \{{\bf -1}, {\bf +1}\} )}\,
\frac{\< V \mb 1\{\Delta^*\}\>_{\mu_\beta}}
{\mu_\beta({\bf 0})} \;,
\end{equation*}
where $V$ is the harmonic function $V(\sigma) = \bb P_\sigma[H_{\bf 0}
< H_{\{{\bf -1}, {\bf +1}\}}]$. By definition of $\theta_\beta$, the
first fraction on the right hand side is bounded. 

It remains to estimate the second fraction, which is equal to
\begin{equation*}
\frac 1{\mu_\beta({\bf 0})} \sum_{\sigma \in \Delta^*_0}  \mu_\beta(\sigma) \bb P_\sigma[H_{\bf 0}
< H_{\{{\bf -1}, {\bf +1}\}}]\;,
\end{equation*}
where $\Delta^*_0 = \{\sigma\in \Delta^* : \mu_\beta(\sigma) \succeq
\mu_\beta({\bf 0})\}$. By \eqref{mf02}, this sum is less than or equal to
\begin{equation*}
\sum_{\sigma \in \Delta^*_0} \frac {\Cap (\sigma , {\bf 0})} {\mu_\beta({\bf 0})} 
\, \frac{\mu_\beta(\sigma)}{\Cap (\sigma, \ms M)}\;.
\end{equation*}
Each term of this sum vanishes as $\beta\uparrow\infty$. Indeed, as
$\sigma$ belongs to $\Delta^*_0$, to reach $\sigma$ from $\bf 0$ the
chain has to escape from the well of $\bf 0$ so that $\Cap(\sigma ,
{\bf 0})/\mu_\beta({\bf 0}) \approx \theta_\beta^{-1}$. On the other
hand, $\mu_\beta(\sigma)/\Cap (\sigma, \ms M)$ is the time scale in
which the process reaches one of the configurations in $\ms M$ starting
from $\sigma$, a time scale of smaller order than the one in which it
jumps between configurations in $\ms M$.

% Est-ce suffisant comme explication. Faut-il etre Plus explicite?

By \eqref{ea1}, to prove the lemma for $\xi={\bf 0}$, we just have to
show that
\begin{equation*}
\lim_{\beta\to\infty} \bb E_{\bf 0} \Big[ \frac 1{\theta_\beta} \int_0^{t\theta_\beta} \mb
1\{\sigma(s) \in \Delta^*\}\, ds \, \mb 1\{H_{\{{\bf -1},
  {\bf +1}\}} \le t\theta_\beta \} \Big]\;=\; 0\;. 
\end{equation*}
Since $\bb P_{\bf 0} [H_{\bf -1} <H_{\bf +1}] \to 0$, we may add in
the previous expectation the indicator of the set $\{H_{\bf +1}
<H_{\bf -1}\}$. Rewrite the integral over the time interval $[0,
t\theta_\beta]$ as the sum of an integral over $[0, H_{\{{\bf -1},
  {\bf +1}\}}]$ with one over the time interval $[H_{\{{\bf -1}, {\bf
    +1}\}}, t\theta_\beta]$. The expectation of the first one is
handled by \eqref{ea1}. The expectation of the second one, by the
strong Markov property, on the set $\{H_{\{{\bf -1}, {\bf +1}\}} \le
t\theta_\beta \} \cap \{H_{\bf +1} <H_{\bf -1}\}$, is less than or
equal to
\begin{equation*}
\bb E_{\bf +1} \Big[ \frac 1{\theta_\beta} \int_0^{t\theta_\beta} \mb
1\{\sigma(s) \in \Delta^*\}\, ds  \Big]\;. 
\end{equation*}
By the first part of the proof this expectation vanishes as
$\beta\uparrow\infty$. 

It remains to consider the case $\xi = {\bf -1}$. As in the case $\xi
= {\bf 0}$, we first estimate the expectation in \eqref{ea1}, with
$H_{\{{\bf 0}, {\bf +1}\}}$ instead of $H_{\{{\bf -1}, {\bf
    +1}\}}$. Then, we repeat the arguments presented for $\xi = {\bf
  0}$, with obvious modifications, to reduce the case $\xi = {\bf -1}$
to the case $\xi = {\bf 0}$, which has already been examined.
\end{proof}

We conclude this section proving the assertion of Remark
\ref{rm3}. Fix $\eta\in \ms M$, $\sigma\in\ms V_\eta$, $\sigma\not =
\eta$.  By \eqref{mf02}, \eqref{mf03} and Assertion \ref{as51},
\begin{equation}
\label{basin}
\bb P_\sigma[H_{\ms M \setminus \{\eta\}} < H_{\eta}] \;\le\;
\frac{\Cap (\sigma , \ms M \setminus \{\eta\})}{\Cap (\sigma, \eta)}
\;\approx\; \sum_{\xi\in \ms M \setminus \{\eta\}}
\frac{\Cap (\sigma , \xi)}{\Cap (\sigma, \eta)}
\;\approx\; \sum_{\xi\in \ms M \setminus \{\eta\}}
\frac{\Cap (\eta , \xi)}{\Cap (\sigma, \eta)}\;\cdot
\end{equation}
By monotonicity of the capacity, the previous expression is bounded by
$2\Cap (\eta , \ms M \setminus \{\eta\})/\Cap (\sigma, \eta)$, which
vanishes in view of Assertion \ref{as52}.  

\smallskip\noindent{\bf Acknowledgements.} The authors wish to thank
O. Benois and M. Mourragui for fruitful discussions.

\end {document}